\documentclass[11pt]{amsart}

\usepackage[margin=1in]{geometry}

\usepackage[T1]{fontenc} 
\usepackage{lmodern} 
\usepackage{amsmath,amssymb,mathtools} 
\usepackage{microtype} 
\usepackage{setspace}

\usepackage{xcolor}
\definecolor{paperlink}{RGB}{0,0,114}
\usepackage[colorlinks=true,citecolor=paperlink,linkcolor=paperlink,urlcolor=paperlink]{hyperref}

\hypersetup{ pdftitle={Bounded symbols and one-component geometry}, pdfauthor={Sam Looi}, 
pdfsubject={Bounded symbols of truncated Toeplitz operators and the characterization of one-component inner functions}, pdfkeywords={truncated Toeplitz operator, bounded symbol, inner function, one-component inner function, connected sublevel set} }

\newtheorem{theorem}{Theorem}[section] 
\newtheorem{lemma}[theorem]{Lemma} 

\newtheorem{proposition}[theorem]{Proposition} 
\newtheorem{corollary}[theorem]{Corollary} 
\theoremstyle{remark} 
\newtheorem{remark}[theorem]{Remark}

\newcommand{\D}{\mathbb D} \newcommand{\T}{\mathbb T} \newcommand{\TT}{\mathcal T} \newcommand{\CLS}{\mathrm{CLS}} \newcommand{\ip}[2]{\left\langle #1,#2\right\rangle} \newcommand{\norm}[1]{\left\lVert #1\right\rVert} \DeclareMathOperator{\arctanh}{arctanh}

\numberwithin{equation}{section} 
\title[The bounded symbol conjecture and one-component geometry]{Bounded symbols and one-component geometry}\thanks{An earlier version of this article, containing all the mathematical results presented here, was made publicly available on \href{https://doi.org/10.5281/zenodo.23214254}{7 October 2026}, both under a permanent DOI and on the author's personal website. Submission to arXiv was deferred to comply with its limit of two new submissions per calendar month. Subsequent revisions have been limited to wording and exposition.}
\author[S. Looi]{Sam Looi} 
\date{11 October 2026}

\address{California Institute of Technology, Pasadena, CA 91125, USA} \email{s.looi9309@gmail.com}

\begin{document}

\begin{abstract}
We prove that every bounded truncated Toeplitz operator on a Hardy model space admits a bounded symbol if and only if its defining inner function is one-component, meaning that it has a connected sublevel set at some level in $(0,1)$. This resolves the bounded symbol conjecture of Baranov, Chalendar, Fricain, Mashreghi and Timotin and, through known equivalences, the embedding and weak factorization conjecture of Baranov, Bessonov and Kapustin. A uniform bound $R$ for the ratio of least bounded symbol norm to operator norm on the rank-one family associated with point evaluation in the disk already forces this geometry. The corresponding bound for all bounded truncated Toeplitz operators is at most $CR\log(e+R)$, where $C$ is absolute, and the logarithmic factor is optimal. In a nonzero de Branges space, the same geometric condition characterizes when every sufficiently small equivalent change of norm preserving the de Branges structure is represented exactly by reweighting the original integral with a positive density bounded above and away from zero. This characterization remains valid when the norm changes are restricted to discrete sampling norms. Outside the one-component class, every model space admits self-adjoint trace class truncated Toeplitz operators without bounded symbols, with arbitrarily small norm and singular values satisfying any prescribed positive upper bounds.
\end{abstract}

\maketitle

\section{Introduction}

A truncated Toeplitz operator is obtained by multiplication by a function, called its symbol, followed by orthogonal projection onto a model subspace of the Hardy space. A bounded symbol always gives a bounded operator, but the converse is a question about representation: different symbols can define the same operator, so an unbounded symbol may have a bounded representative. A model space has the \emph{bounded symbol property} if every bounded truncated Toeplitz operator on it admits an essentially bounded symbol. This is the property underlying Sarason's bounded symbol question~\cite[Section~1, p.~493]{Sarason}.

We resolve the bounded symbol conjecture of Baranov, Chalendar, Fricain, Mashreghi and Timotin: a model space has the bounded symbol property if and only if its defining inner function is one-component, meaning that it has a connected sublevel set at some level strictly between zero and one. Through the equivalences of Baranov, Bessonov and Kapustin~\cite[Theorem~2.4]{BBK}, this also resolves their embedding and weak factorization conjecture. In this sense, the inner function's geometry determines whether the boundedness of every truncated Toeplitz operator can be detected by boundedness of the multiplication symbol defining the operator. Only one sublevel set needs to be connected; lower sublevel sets may still have several components.

To prove the new implication, we only need to test a family of rank-one operators associated with point evaluation in the disk. Uniform control of their symbol norms relative to their operator norms gives a lower bound on the hyperbolic derivative where the values of the inner function are close to the unit circle. This bound controls the lengths of inverse paths, and a flow recovers connectedness. That is, norm estimates on an explicit, small family of operators determine connectedness of the inner function.

\subsection{Model spaces and bounded symbols}

The unit disk is $\D=\{z\in\mathbb C:|z|<1\}$, and its boundary circle is $\T=\partial\D$. We use normalized Lebesgue measure $m$ on $\T$, so that $dm(e^{it})=dt/(2\pi)$. The Hardy space $H^2$ consists of analytic functions $f(z)=\sum_{n\ge0}a_nz^n$ with $\norm f_2^2=\sum_{n\ge0}|a_n|^2<\infty$. Such functions have boundary values in $L^2(m)$, with the same norm, and we use this identification. Our inner products are linear in the first variable.

We write $H^\infty$ for the bounded analytic functions on $\D$. A nonconstant bounded analytic function $\Theta$ is called \emph{inner} if its boundary values satisfy $|\Theta|=1$ almost everywhere on $\T$. Multiplication by $\Theta$ is an isometry on $H^2$. The orthogonal complement of its range is the \emph{model space}
\[
 K_\Theta=H^2\cap(\Theta H^2)^\perp.
\]
We write $P_\Theta:L^2(m)\to K_\Theta$ for the orthogonal projection.

For example, $K_{z^2}=\operatorname{span}\{1,z\}$. More generally,
\begin{equation*}
 K_{z^N}=\operatorname{span}\{1,z,\ldots,z^{N-1}\},\qquad N\ge1\,.
\end{equation*}

Multiplication by a \emph{symbol} $\varphi\in L^2(m)$, followed by projection onto $K_\Theta$, defines the \emph{truncated Toeplitz operator}
\[
 A_\varphi^\Theta f=P_\Theta(\varphi f),
 \qquad f\in K_\Theta\cap H^\infty.
\]
The restriction to bounded $f$ ensures that $\varphi f\in L^2(m)$. We write $\TT(\Theta)$ for the operators defined this way that extend boundedly to all of $K_\Theta$. If $M_\varphi$ denotes multiplication by $\varphi$, then each element of $\TT(\Theta)$ is the bounded extension of $P_\Theta M_\varphi$ from this domain.

For example, with $\varphi=2+z+\bar z$, projection onto $K_{z^2}$ discards the $\bar z$ term from $\varphi\cdot1$ and the $z^2$ term from $\varphi\cdot z$, giving
\[
 A_\varphi^{z^2}1=2+z,\qquad A_\varphi^{z^2}z=1+2z,
 \qquad [A_\varphi^{z^2}]_{\{1,z\}}=
 \begin{pmatrix}2&1\\1&2\end{pmatrix}.
\]
More generally, the truncated Toeplitz operators on $K_{z^N}$ are matrices constant along each diagonal. The symbol $\varphi+z^2+\bar z^2$ gives the same operator on $K_{z^2}$: those extra terms disappear under projection.

A \emph{bounded symbol} for $A\in\TT(\Theta)$ is a function $\psi\in L^\infty(m)$ with $A=A_\psi^\Theta$. The bounded symbol property requires such a representative for every $A\in\TT(\Theta)$, although the original symbol need not itself be bounded. Such representatives always exist in finite dimensions.

For comparison, bounded classical Toeplitz operators on $H^2$ have bounded symbols with equality of the two norms \cite{BH}. Earlier positive cases for truncated operators include Sarason's bounded analytic representatives for bounded operators with analytic symbols \cite{Sarason1967} and Rochberg's theorem for the Paley-Wiener space \cite{Rochberg}.

Sarason's question has a negative answer in general: Baranov, Chalendar, Fricain, Mashreghi and Timotin constructed bounded truncated Toeplitz operators without bounded symbols and conjectured that the bounded symbol property holds exactly for one-component inner functions \cite[Theorem~5.3 and Remark~7.3]{BCFMT}.

An inner function is \emph{one-component}, written $\Theta\in(\CLS)$, if
\[
 \Omega_t(\Theta)=\{z\in\D:|\Theta(z)|<t\}
\]
is connected for some $0<t<1$.

Cohn introduced this class through embedding measures \cite{Cohn}; Aleksandrov characterized it by uniform reproducing kernel estimates \cite{Aleksandrov}. For this class, Cohn's work \cite{Cohn1986} also gives bounded symbols for bounded truncated Toeplitz operators with nonnegative $L^2$ symbols; see \cite[Theorem 7.1]{BCFMT}.

Baranov, Bessonov and Kapustin proved that the property is equivalent to conditions on embedding measures and weak factorization \cite[Theorem~2.4]{BBK}. Their Corollary~2.5 establishes sufficiency of the one-component condition, and their Conjecture~2.6 states the remaining necessity direction \cite{BBK}.

The \emph{reproducing kernel} $k_w^F$ represents point evaluation: $f(w)=\ip{f}{k_w^F}$ for $f\in K_F$. For a nonconstant inner $F$, it and its \emph{conjugate kernel} are
\[
 k_w^F(z)=\frac{1-\overline{F(w)}F(z)}{1-\overline w z},
 \qquad \widetilde k_w^F(z)=\frac{F(z)-F(w)}{z-w},
\]
with value $F'(w)$ at $z=w$ in the second quotient. Their relation is explained in Subsection~\ref{subsec:least-symbol}. For nonzero vectors $f,g$, we use the operator of rank one $(f\otimes g)h=\ip{h}{g}f$. We put $R_w=\widetilde k_w^\Theta\otimes k_w^\Theta$. Thus $R_wf=f(w)\widetilde k_w^\Theta$: the operator depends on its input only through its value at $w$. Each $R_w$ has a bounded symbol \cite[Theorem 5.1]{Sarason}. The issue is whether the symbol norms can be bounded relative to $\norm{R_w}$ uniformly in $w$.

We measure the least bounded symbol norm by
\[
 b_\Theta(A)=\inf\{\norm\psi_\infty:
       \psi\in L^\infty(m),\ A_\psi^\Theta=A\},
\]
with value $+\infty$ when no bounded symbol exists, and put
\begin{equation}\label{eq:symbol-constants}
 M_*(\Theta)=\sup_{0\ne A\in\TT(\Theta)}
       \frac{b_\Theta(A)}{\norm A},\qquad
 R_*(\Theta)=\sup_{w\in\D}
       \frac{b_\Theta(R_w)}{\norm{R_w}}.
\end{equation}
We have $\norm{A_\psi^\Theta}\le\norm\psi_\infty$, and $1\le R_*\le M_*$. $R_*$ is the rank-one symbol constant, while $M_*$ is the general symbol constant.

The bounded symbol property is equivalent to $M_*(\Theta)<\infty$, by the open mapping theorem; see the proof of Theorem~\ref{thm:main}.

\begin{theorem}[Characterization of one-component model spaces] \label{thm:main} For a nonconstant inner function $\Theta$, the following are equivalent: \begin{enumerate} \item[(i)] $\Theta$ is one-component. \item[(ii)] Every operator in $\TT(\Theta)$ has a bounded symbol. \item[(iii)] $R_*(\Theta)<\infty$. \end{enumerate} \end{theorem}

The implication (i) $\Rightarrow$ (ii) is due to Baranov, Bessonov and Kapustin \cite[Corollary~2.5]{BBK}. We prove (iii) $\Rightarrow$ (i): if $b_\Theta(R_w)/\norm{R_w}$ is bounded uniformly in $w$, then $\Theta$ has a connected sublevel set.

We will use the Blaschke product of degree two $\Theta_a(z)=(a-z^2)/(1-az^2)$, $0<a<1$, whose zeros are $\pm\sqrt a$. Its only critical point is zero, with critical value $a$.

The sublevel set $\Omega_t(\Theta_a)$ has two components for $0<t\le a$ and is connected for $a<t<1$. Thus a one-component inner function can have disconnected sublevel sets at lower levels. We give the calculation in Subsection~\ref{sec:geometry}.

We also compare the best symbol bound for this rank-one family with the best bound for all bounded truncated Toeplitz operators. Theorem~\ref{thm:sharp-rank-one} makes this reduction quantitative: whenever $R_*(\Theta)<\infty$,
\[
 R_*(\Theta)\le M_*(\Theta)
 \le C R_*(\Theta)\log(e+R_*(\Theta)),
\]
where $C$ is an absolute constant. The logarithmic factor is necessary: our $N$-dimensional examples have $R_*$ of order $N/\log N$ and $M_*$ of order $N$.

The same geometric condition answers two questions in harmonic analysis. Through the equivalences of Baranov, Bessonov and Kapustin~\cite[Theorem~2.4]{BBK}, it determines when the measures giving bounded embeddings at exponents $1$ and $2$ coincide for the model spaces associated with the square of the inner function, and when the corresponding subspace of $H^1$ admits weak factorization into products of model space functions; see Theorem~\ref{thm:full-characterization}. These are quantitative statements: the factorization controls the sum of the factors' $H^2$ norm products by the $H^1$ norm of the original function.

In de Branges spaces, the corresponding question is whether changing the norm of a space of entire functions amounts to reweighting its original integral. Such norm changes arise in spectral theory~\cite[Section~7.5]{PoltoratskiOrder}. Equivalence of two norms compares their sizes but clearly does not provide a bounded density representing one relative to the other. We show that every sufficiently small equivalent change of norm preserving the de Branges structure has an exact representation by a positive density bounded above and away from zero if and only if the associated inner function is one-component. It suffices to require this representation for every discrete sampling norm sufficiently close to the original norm. In the one-component case, if a weighted sum of squared point values approximates the original squared norm with sufficiently small relative error uniformly over the space, then we can choose a single positive weight, bounded above and away from zero, so that the reweighted integral equals this sum for every function in the space.

On every model space outside the one-component class, we construct self-adjoint trace class truncated Toeplitz operators without bounded symbols, with arbitrarily small norm and any prescribed positive upper bounds on their singular values. Thus no rate of singular value decay ensures the existence of a bounded symbol on such a space. The corresponding de Branges norms may be arbitrarily close to the original norm and still have no bounded representing density.

The search for symbol criteria also motivated our earlier work on Bargmann--Fock space~\cite{LooiBC}, where the heat transform conditions proposed in the Berger--Coburn boundedness conjecture~\cite{BergerCoburn} and the compactness conjecture of Bauer, Coburn and Isralowitz~\cite[Conjecture~2]{BCI} fail to characterize the corresponding operator properties. We discuss the connection with our compact examples in Remark~\ref{rem:BC-connection}.

\subsection{Equivalent conditions}
We state the full set of equivalences, connecting the geometry of $\Theta$ with operator symbols, embedding measures and weak factorization. More precisely, we give conditions on operator symbols, embedding measures and weak factorization that are equivalent to $\Theta$ being one-component.

For a holomorphic map $G:\D\to\D$, we write
\[
 \eta_G(z)=\frac{(1-|z|^2)|G'(z)|}{1-|G(z)|^2}
\]
for its \emph{hyperbolic derivative}. The Schwarz--Pick lemma gives $0\le\eta_G\le1$.

The classes $\mathcal C_p(\Phi)$ consist of measures $\mu$ for which the inclusion $K_\Phi^p\hookrightarrow L^p(|\mu|)$ is bounded; precise definitions are in Section~\ref{sec:model-space-consequences}. 

If $f,g\in K_\Theta$, then $fg$ belongs to $H^1\cap\bar z^{\,2}\Theta^2\overline{H^1}$, where we identify Hardy space functions with their boundary values on $\T$. The weak factorization statement in Theorem~\ref{thm:full-characterization}(vii) says that every $F$ in this space is a sum of four such products. The factors can be chosen so that $\sum_{j=1}^4\norm{f_j}_2\norm{g_j}_2\le C_\Theta\norm F_1$, with a constant independent of $F$.

\begin{theorem}[Geometric and analytic characterizations]\label{thm:full-characterization} 
For a nonconstant inner function $\Theta$, the following are equivalent: \begin{enumerate} \item[(i)] $\Theta$ is one-component. \item[(ii)] Every operator in $\TT(\Theta)$ has a bounded symbol. \item[(iii)] $R_*(\Theta)<\infty$. \item[(iv)] Every self-adjoint trace class operator in $\TT(\Theta)$ has a bounded symbol. \item[(v)] (Quadratic criterion) $\displaystyle \sup_{z\in\D}\frac{1-\eta_\Theta(z)^2}{(1-|\Theta(z)|^2)^2}<\infty$. \item[(vi)] $\mathcal C_1(\Theta^2)=\mathcal C_2(\Theta^2)$. \item[(vii)] There is $C_\Theta>0$ such that every $F\in H^1\cap\bar z^{\,2}\Theta^2\overline{H^1}$ admits
\[
 F=\sum_{j=1}^4 f_jg_j,\qquad f_j,g_j\in K_\Theta,\qquad
 \sum_{j=1}^4\norm{f_j}_2\norm{g_j}_2\le C_\Theta\norm F_1.
\]
\item[(viii)] Every compact operator in $\TT(\Theta)$ has a symbol in $\Theta C(\T)$. 
\end{enumerate} \end{theorem}

The equivalences of (ii) with (vi)--(vii) are due to \cite[Theorem~2.4 and Proposition~4.1]{BBK}; that with (viii) is due to \cite[Theorem~3.2]{OLoughlin}. Theorems~\ref{thm:main} and~\ref{thm:quadratic} supply (i), (iii) and (v). For (iv), Theorem~\ref{thm:arbitrary-decay} proves more: outside the one-component class, self-adjoint trace class operators without bounded symbols exist with arbitrarily small norm and any prescribed positive singular value bounds. Thus, on such a model space, arbitrarily rapid decay of singular values does not guarantee a bounded symbol.

The quadratic criterion connects the topology of a sublevel set with a uniform quadratic rate of approach to equality in the Schwarz--Pick lemma as $|\Theta(z)|\to1$. The exponent is optimal, already for $\Theta(z)=z^2$; see Theorem~\ref{thm:quadratic}.

Weak factorization explains how operator norm bounds yield bounds for symbols: it gives an $L^1$ estimate for the functional associated with a bounded truncated Toeplitz operator, from which Hahn--Banach and $L^1$--$L^\infty$ duality produce a bounded symbol \cite[Theorem~2.3]{BBK}. The quantitative form of this argument is used in Appendix~\ref{app:annular-symbol}.

\subsection{Proof strategy}

We explain how the assumption $R:=R_*(\Theta)<\infty$ forces a connected sublevel set, proving (iii) $\Rightarrow$ (i) in Theorem~\ref{thm:main}. The symbol estimate gives a positive lower bound for the hyperbolic derivative where $|\Theta|$ is close to one. We use this bound to construct a continuous retraction of the disk onto a closed sublevel set, which shows it is connected.

We use the hyperbolic metric on $\D$, in which an infinitesimal displacement $dz$ at $z$ has length
\[
 \frac{|dz|}{1-|z|^2}.
\]
A holomorphic map $G:\D\to\D$ sends this displacement to $G'(z)\,dz$ at $G(z)$. The ratio of the image length to the original length is $\eta_G(z)$. Thus $\eta_G$ plays the role of $|G'|$ when lengths are measured in the hyperbolic metric rather than the Euclidean metric. The Schwarz--Pick bound $\eta_G\le1$ says that holomorphic disk maps cannot increase hyperbolic lengths. Disk automorphisms preserve these lengths and have $\eta_G\equiv1$.

Under condition (ii) of Theorem~\ref{thm:main}, the rank-one argument of \cite[Theorem~5.6]{BCFMT} bounds each kernel's $L^p$ norm relative to its squared $L^2$ norm, uniformly in the evaluation point for each fixed finite $p>2$. That argument did not establish the corresponding $L^\infty$ bound, which characterizes the one-component property \cite[Theorem~1.2]{Aleksandrov}.

We first compute the least $L^2$ symbol norm for a rank-one test operator by orthogonal projection. Since $m$ is normalized, every bounded symbol $\psi$ satisfies $\norm\psi_2\le\norm\psi_\infty$, so the minimum is a lower bound for the $L^\infty$ norm of every bounded representative. A disk automorphism moves the evaluation point to zero without changing the normalized rank-one symbol estimate (Lemma~\ref{lem:automorphisms}). The exact minimum in Proposition~\ref{prop:least-symbol} depends on the value and derivative of the transformed inner function at zero.

Returning to the original evaluation point and comparing with the uniform symbol bound, we obtain
\begin{equation*}
 \eta_\Theta(w)^2\ge 1-R^2\frac{1-|\Theta(w)|^2}{1+|\Theta(w)|^2}
 \qquad(w\in\D)\,.
\end{equation*}
For fixed $R$, the right-hand side tends to one as $|\Theta(w)|\to1$. Thus the symbol estimate prevents severe contraction of hyperbolic lengths in this region. In particular, we can choose $0<r<1$ and $c>0$ such that $\eta_\Theta\ge c$ on $\{|\Theta|>r\}$; the calculation is given in Proposition~\ref{prop:rank-one-derivative}.

To recover connectedness, we fix $r<s<1$ and move points into $E_s:=\{z\in\D:|\Theta(z)|\le s\}$. Points already in $E_s$ stay fixed. For $z\notin E_s$, the flow
\begin{equation*}
 \gamma_z'(u)=-\frac{\Theta(\gamma_z(u))}{\Theta'(\gamma_z(u))},
 \qquad \gamma_z(0)=z
\end{equation*}
is defined locally because $\Theta'$ does not vanish on $\{|\Theta|>r\}$. Along the trajectory, the chain rule gives
\begin{equation*}
 \Theta(\gamma_z(u))=e^{-u}\Theta(z)\,,
\end{equation*}
so its image moves radially inward. It remains to show that the lifted trajectory exists until it reaches level $s$, at time $\log(|\Theta(z)|/s)$. The estimate $\eta_\Theta\ge c$ bounds the hyperbolic length of the lifted trajectory by $c^{-1}$ times that of its radial image, which is finite. If the solution ceased to exist before reaching $E_s$, completeness would give a limit point in $\D$. At this point $|\Theta|\ge s>r$, so $\Theta'\ne0$ and the solution could be continued, so the trajectory reaches $E_s$.

Stopping each trajectory at level $s$ gives a continuous retraction of the disk onto $E_s$, so $E_s$ is connected. For $s<t<1$, every point of $\Omega_t(\Theta)$ can be joined to $E_s$ along one of these trajectories, which stays in $\Omega_t(\Theta)$ because $|\Theta|$ decreases. Hence $\Omega_t(\Theta)$ is connected as well. The construction applies more generally to any holomorphic self-map $G$ of the disk with $\eta_G\ge c>0$ on $\{|G|>r\}$ for some $0\le r<1$: every sublevel set $\{|G|<t\}$ with $r<t<1$ is contractible. Lemma~\ref{lem:geometry} gives the full statement and proof in Section~\ref{sec:symbol}.

For the quantitative bound, Section~\ref{sec:quantitative} sharpens the derivative estimate; Appendix~\ref{app:annular-symbol} bounds the general symbol constant by a constant times the reciprocal width of a covering annulus.

\subsection{De Branges spaces and sampling}

We also apply our characterization to de Branges spaces $\mathcal H(E)$, Hilbert spaces of entire functions with norm
$$
\norm{F}_E^2=\int_{\mathbb R}\frac{|F(x)|^2}{|E(x)|^2}\,dx.
$$
Here $E$ has no real zeros, and the associated inner function in the upper half plane is $\Theta_E(z)=\overline{E(\bar z)}/E(z)$. The choice $E(z)=e^{-iaz}$, $a>0$, gives the classical Paley--Wiener space of entire extensions of square integrable bandlimited functions. The defining conditions on $E$ and the identification with a model space are given in Section~\ref{sec:spectral-applications}.

Changes of norm on a fixed space of entire functions arise in spectral theory \cite[Section~7.5]{PoltoratskiOrder}. We call an equivalent norm $\norm{\cdot}_*$ on $\mathcal H(E)$ \emph{admissible} if $\mathcal H(E)$, equipped with this norm and with the same entire functions, is again a de Branges space. 

For $\mathcal H(E)\ne\{0\}$, an equivalent Hilbert norm is represented by a bounded positive invertible operator $A$ on the associated model space. Admissibility is precisely the condition that $A$ be a truncated Toeplitz operator (Proposition~\ref{prop:spectral-norm-correspondence}). Thus, for an admissible norm with $\Theta_E$ one-component, $A-I$ has a bounded symbol. Since $A-I$ is self-adjoint, we can take the symbol $b$ to be real, and
\begin{equation*}
 \norm{F}_*^2-\norm{F}_E^2
 =\int_{\mathbb R}b(x)\frac{|F(x)|^2}{|E(x)|^2}\,dx
 \qquad(F\in\mathcal H(E)).
\end{equation*}
The weight $1+b$ therefore represents the new norm. By the uniform symbol estimate, we can choose $b$ with $\norm b_\infty<1$ when $\norm{A-I}$, the relative change in the squared norm, is sufficiently small. This gives a weight bounded above and away from zero.

We ask whether an admissible norm can be represented by reweighting the original integral:
$$
\norm{F}_*^2
=\int_{\mathbb R}w(x)\frac{|F(x)|^2}{|E(x)|^2}\,dx
\qquad(F\in\mathcal H(E)).
$$
Here, one uses every $F\in\mathcal H(E)$. A real measurable function $w$ satisfying this identity is called a \emph{representing density}; we seek one with $0<c\le w\le C<\infty$ almost everywhere. 

For $\mathcal H(E)\ne\{0\}$, Theorem~\ref{thm:spectral-dichotomy} shows that every sufficiently small admissible norm change has a representing density with these bounds if and only if $\Theta_E$ is one-component. In the one-component case, a sufficiently small relative change $d$ in the squared norm admits a representing density $w$ satisfying $\|w-1\|_\infty\le C_Ed$, where $C_E$ depends only on the original space.

If $\Theta_E$ is not one-component, there are arbitrarily small admissible norm changes with no bounded representing density. These counterexamples can be chosen in the form $\norm{F}_*^2=\ip{(I+K)F}{F}_E$, where $K$ is self-adjoint and trace class, has arbitrarily small operator norm, and satisfies any prescribed positive (upper) bounds on its singular values.

In the application to sampling, we state an equivalent criterion in terms of discrete norms. Corollary~\ref{cor:discrete-spectral-weights} shows that $\Theta_E$ is one-component if and only if every discrete sampling norm $\norm{F}_*^2=\sum_n\gamma_n|F(\lambda_n)|^2$ sufficiently close to the original norm admits a representing density with $0<c\le w\le C<\infty$. Here $\lambda_n\in\mathbb R$ and $\gamma_n>0$. In the one-component case, the conclusion converts the sampling inequalities into an exact integral identity, with one representing density valid throughout the space. The converse uses de Branges' theorem on orthogonal bases of reproducing kernels, which gives a discrete sampling representation for every admissible norm. The connection with the sampling and interpolation results of Marzo, Nitzan and Olsen \cite{MNO} comes from their phase criterion for the one-component property; we refer to Remark~\ref{cor:phase}.

\section{From rank-one symbols to one-component geometry} \label{sec:symbol}

\subsection{Proof of the main characterization}

We begin by stating the operator estimate and the geometric lemma, and use them to prove Theorem~\ref{thm:main}. Their proofs then follow in the rest of this section. The operator estimate controls how much $\Theta$ can contract hyperbolic lengths where $|\Theta|$ is close to one.

\begin{proposition}[Rank-one control of the hyperbolic derivative] \label{prop:rank-one-derivative} For a nonconstant inner function $\Theta$, assume $R=R_*(\Theta)<\infty$. Then, for every $w\in\D$,
\begin{equation}\label{eq:main-derivative}
 \eta_\Theta(w)^2\ge
 1-R^2\frac{1-|\Theta(w)|^2}{1+|\Theta(w)|^2}.
\end{equation}
\end{proposition}

The right-hand side approaches one as $|\Theta(w)|\to1$. The next lemma turns a positive lower bound in this region into connectedness, which is all we need for the characterization. We write $d_\D$ for distance in the complete metric $|dz|/(1-|z|^2)$; thus $d_\D(0,z)=\arctanh|z|$.

\begin{lemma}[Flow to a sublevel set]\label{lem:geometry} For a holomorphic map $G:\D\to\D$, assume $\eta_G\ge c>0$ on $\{|G|>r\}$, where $0\le r<1$. Then every $\Omega_t(G)=\{|G|<t\}$ with $r<t<1$ is nonempty and contractible, in particular connected. More precisely, for $r<s<1$, the set $E_s=\{|G|\le s\}$ is a nonempty strong deformation retract of $\D$ and is contractible.\footnote{We recall the following definitions: (a) A topological space is contractible if the identity map can be continuously deformed to a constant map; (b) a strong deformation retract moves each point continuously into $E_s$ while fixing points of $E_s$.} For $z\notin E_s$,
\begin{equation}\label{eq:distance}
\Delta_s(z)\le d_\D(z,E_s)\le c^{-1}\Delta_s(z),
\qquad
\Delta_s(z)=\arctanh|G(z)|-\arctanh s.
\end{equation}
\end{lemma}

\begin{proof}[Proof of Theorem~\ref{thm:main}] The implication (i) $\Rightarrow$ (ii) is \cite[Corollary~2.5]{BBK}. For (ii) $\Rightarrow$ (iii), Sarason's theorem on closedness in the weak operator topology makes $\TT(\Theta)$ a Banach space in the operator norm \cite[Theorem~4.2]{Sarason}. Under (ii), the contraction $Q_\Theta:L^\infty(m)\to\TT(\Theta)$, $Q_\Theta\psi=A_\psi^\Theta$, is surjective. The open mapping theorem gives $M\ge1$ such that every $A\in\TT(\Theta)$ has a symbol satisfying
\begin{equation}\label{eq:uniform}
\norm{\psi}_\infty\le M\norm{A}.
\end{equation}
In particular, $R_*(\Theta)\le M$.

For (iii) $\Rightarrow$ (i), we put $R=R_*(\Theta)<\infty$ and choose $0<r<1$ close enough to one that $c^2:=1-R^2(1-r^2)/(1+r^2)>0$. Proposition~\ref{prop:rank-one-derivative} gives
\[
 \eta_\Theta(z)\ge c\qquad (|\Theta(z)|>r).
\]
Lemma~\ref{lem:geometry} then makes $\Omega_t(\Theta)$ connected for $r<t<1$, so $\Theta$ is one-component. \end{proof}

\subsection{Symbol norms under disk automorphisms}
A null symbol gives the zero operator, so adding one to a symbol does not change the operator. The following characterization is due to Izuchi, as Sarason acknowledges \cite[p.~494 and Theorem 3.1]{Sarason}:
\begin{equation}\label{eq:null-symbol}
A_\varphi^\Theta=0
\quad\Longleftrightarrow\quad
\varphi\in\Theta H^2+\overline{\Theta H^2},
\qquad \varphi\in L^2(m).
\end{equation}
No normalization of $\Theta(0)$ is needed.

For a one-component $\Theta$, the bounded symbol theorem \cite[Corollary~2.5]{BBK} and the null symbol identity \eqref{eq:null-symbol} give, for every $\varphi\in L^2(m)$,
$$
A_\varphi^\Theta\text{ is bounded}
\quad\Longleftrightarrow\quad
\varphi\in L^\infty(m)+\Theta H^2+\overline{\Theta H^2}.
$$
Theorem~\ref{thm:main} shows that one-component model spaces are exactly those on which this criterion holds for every $L^2$ symbol.

The least bounded symbol norm $b_\Theta$ is homogeneous and subadditive and, for each finite $M\ge0$, the sublevel set
\begin{equation}\label{eq:cost-sublevels}
 \{A\in\TT(\Theta):b_\Theta(A)\le M\}
\end{equation}
is closed in operator norm. Also, whenever $b_\Theta(A)$ is finite, there is a symbol whose norm equals $b_\Theta(A)$; indeed, the map $\psi\mapsto A_\psi^\Theta$ sends each weak-* compact ball of $L^\infty(m)$ continuously into the weak operator topology, since
\[
 \langle A_\psi^\Theta f,g\rangle
       =\int_{\T}\psi f\overline g\,dm,
       \qquad f,g\in K_\Theta,
\]
and $f\overline g\in L^1(m)$. Its image is therefore compact in the weak operator topology, and hence closed in the operator norm. Moreover, weak-* compactness applied to a minimizing sequence gives attainment.

To prove Proposition~\ref{prop:rank-one-derivative}, we first move the point $w$ to the origin. The unitary conjugation associated with disk automorphisms in \cite[Proposition~4.1]{CGRW} and \cite[Theorem~6.1]{Sarason2012} preserves operator norms. We check that it also preserves least bounded symbol norms, so that a calculation at zero gives the estimate everywhere.

\begin{samepage}
\begin{lemma}[Invariance under disk automorphisms]\label{lem:automorphisms} 
Let
\[
\tau_w(z)=\frac{w-z}{1-\overline w z},\qquad w\in\D.
\]
Put $F=\Theta\circ\tau_w$. Then $M_*(F)=M_*(\Theta)$ and $R_*(F)=R_*(\Theta)$, allowing the value $+\infty$. In particular, any uniform bounded symbol estimate for $\Theta$ holds for $F$ with the same constant. \end{lemma}
\end{samepage}

\begin{proof} We put $h_w(z)=\sqrt{1-|w|^2}/(1-\overline w z)$ and $U_wf=h_w(f\circ\tau_w)$. Boundary change of variables and $h_w(h_w\circ\tau_w)=1$ show that $U_w$ is a unitary involution on $L^2(m)$. It maps $H^2$ onto itself and $\Theta H^2$ onto $FH^2$ where $F:= \Theta \circ \tau_w$, so $U_wK_\Theta=K_F$ and $U_wP_\Theta U_w=P_F$. For fixed $w$, the boundary Jacobian is bounded above and below, and $h_w\in H^\infty$. Thus composition preserves $L^2(m)$, and $U_w$ maps the defining domains bijectively. Conjugation of multiplication gives
\begin{equation}\label{eq:covariance}
U_wA_\varphi^\Theta U_w=A_{\varphi\circ\tau_w}^{F},
\qquad\varphi\in L^2(m).
\end{equation}
Composition with $\tau_w$ preserves the essential supremum norm. The identity \eqref{eq:covariance}, applied in both directions, then gives $b_F(U_wAU_w)=b_\Theta(A)$, including when no bounded symbol exists. Since $U_w$ is unitary, $M_*(F)=M_*(\Theta)$. For the rank-one family, $U_wR_\lambda^\Theta U_w=-h_w(\lambda)^2R_{\tau_w(\lambda)}^F$, where superscripts specify the model space. Homogeneity gives $R_*(F)=R_*(\Theta)$ as well. \end{proof}

\subsection{The least $L^2$ symbol norm} \label{subsec:least-symbol}

We first minimize the $L^2$ norm of a symbol, where orthogonal projection is available. Since $m$ is normalized, every bounded symbol satisfies $\norm\varphi_2\le\norm\varphi_\infty$. We compute this minimum in terms of the value and derivative of the inner function at zero, and compare it with the $L^\infty$ symbol bound to estimate the hyperbolic derivative from below. Replacing $\Theta$ by $\Theta\circ\tau_w$ leaves the rank-one symbol constant unchanged by Lemma~\ref{lem:automorphisms}, so the same calculation gives the estimate at any $w\in\D$.

For a nonconstant inner function $F$, the map $C_Fg=F\overline z\,\overline g$ on boundary values is a conjugation on $K_F$: it is conjugate linear, preserves the norm, and satisfies $C_F^2=I$. It sends $k_w^F$ to $\widetilde k_w^F$, explaining the name conjugate kernel. We write $a=F(0)$.

\begin{proposition}[Exact least symbol norm]\label{prop:least-symbol} With $D=1-|a|^4$ and $R_F=\widetilde k_0^F\otimes k_0^F$, $\norm{R_F}=1-|a|^2$ and
\begin{equation}\label{eq:least-symbol}
\min\{\norm{\varphi}_2^2:A_\varphi^F=R_F,\ \varphi\in L^2(m)\}
=D-\frac{4|a|^2|F'(0)|^2}{D}.
\end{equation}
\end{proposition}

Since $\Theta_a(0)=a$, $\Theta_a'(0)=0$ and $\norm\varphi_\infty\ge\norm\varphi_2$, the proposition gives
\[
 R_*(\Theta_a)\ge\frac{\sqrt{1-a^4}}{1-a^2}
 =\sqrt{\frac{1+a^2}{1-a^2}}\to\infty
 \qquad(a\to1).
\]
Thus a critical value approaching 1 forces large symbol constants even in dimension two \cite[Sections~5 and~7]{Sarason2012}. We note that related lower bounds for operator norms appear in \cite[Section~3]{GarciaRoss}.

The derivative appears in the coefficient $2F(0)F'(0)$ of $z$ in $F^2$. The initial symbol $\overline zF$ is already orthogonal to $FH^2$ and we first correct it so that its pairing against $\overline{Fh}$ depends only on $h(0)$. Then, with a second correction, we cancel the remaining term. At both steps, we keep the symbol orthogonal to $FH^2$.

\begin{proof} Sarason's rank-one formula gives $R_F=A_{\overline zF}^{F}$ \cite[Theorem 5.1]{Sarason}. Since $\norm{k_0^F}_2^2=\norm{\widetilde k_0^F}_2^2=1-|a|^2$, the asserted operator norm follows.

By \eqref{eq:null-symbol}, we can subtract any element of $FH^2+\overline{FH^2}$ from the symbol without changing the operator. The least norm representative is the one orthogonal to this subspace. The projection formulas in \cite[Lemmas 5.1 and 3.1, equation (3.3)]{BCFMT} lead to the choice
\[
N_F=FH^2+\overline{FH^2},\qquad
u=\overline z(F-a^2\overline F),\qquad
v=\overline F-\overline a^{\,2}F.
\]
Here $\overline zF-u=a^2\overline{zF}\in N_F$ and $v\in N_F$. For $h\in H^2$, the identity $|F|=1$ on $\T$ gives
\begin{align*}
\ip{u}{Fh}
&=\int_\T\overline z(1-a^2\overline F^{\,2})\overline h\,dm=0,\\
\ip{v}{Fh}
&=\int_\T(\overline F^{\,2}-\overline a^{\,2})\overline
 h\,dm=0,\\
\ip{u}{\overline{Fh}}
&=\int_\T\overline z(F^2-a^2)h\,dm=2aF'(0)h(0),\\
\ip{v}{\overline{Fh}}
&=\int_\T(1-\overline a^{\,2}F^2)h\,dm=Dh(0).
\end{align*}
These identities show that
\[
s_F=u-\frac{2aF'(0)}{D}v
\]
is a symbol of $R_F$ orthogonal to $N_F$. Every other symbol differs from $s_F$ by an element of $N_F$, so $s_F$ has least norm. Finally, $\int_\T F^2\,dm=a^2$ gives
\[
\norm{u}_2^2=\norm{v}_2^2=D,\qquad
\ip{u}{v}=2aF'(0).
\]
Thus $\norm{s_F}_2^2=D-|2aF'(0)|^2/D$, as claimed. \end{proof}

\subsection{Rank-one control of the hyperbolic derivative}

\begin{proof}[Proof of Proposition~\ref{prop:rank-one-derivative}] We put $R=R_*(\Theta)$. For $w\in\D$, we set $F=\Theta\circ\tau_w$, $x=|F(0)|^2=|\Theta(w)|^2$ and $\eta=\eta_\Theta(w)$. Lemma~\ref{lem:automorphisms} gives $R_*(F)=R$, so $b_F(R_F)\le R\norm{R_F}=R(1-x)$. Since $\norm\varphi_2\le\norm\varphi_\infty$, Proposition~\ref{prop:least-symbol} gives
\[
1-x^2-\frac{4x|F'(0)|^2}{1-x^2}\le R^2(1-x)^2.
\]
The identity $|F'(0)|=(1-|w|^2)|\Theta'(w)|=\eta(1-x)$, followed by division by $1-x$, yields
\begin{equation}\label{eq:eta-exact}
1+x-\frac{4x\eta^2}{1+x}\le R^2(1-x).
\end{equation}
By Schwarz-Pick, $0\le\eta\le1$, and
\[
1+x-\frac{4x\eta^2}{1+x}
=(1+x)(1-\eta^2)+\frac{(1-x)^2}{1+x}\eta^2
\ge(1+x)(1-\eta^2).
\]
This proves \eqref{eq:main-derivative}. \end{proof}

\subsection{The flow to a sublevel set}\label{sec:geometry}

Throughout this argument, $r<s<1$ is fixed. For $G(z)=z^2$, the trajectory $\gamma_z(u)=e^{-u/2}z$, stopped when $|\gamma_z(u)|=\sqrt s$, gives a radial retraction onto $E_s=\{|G|\le s\}$ (since $u$ is increasing); points already in $E_s$ stay fixed. For general $G$, requiring $G(\gamma_z(u))=e^{-u}G(z)$ gives $\gamma'=-G/G'$. The derivative bound controls the hyperbolic length of the trajectory, and completeness prevents it from escaping before reaching $E_s$. For inner functions, each open sublevel component is already simply connected \cite[Lemma~2]{CM}; the flow connects the whole sublevel set.

\begin{proof}[Proof of Lemma~\ref{lem:geometry}] Since $G'$ does not vanish on $V=\{|G|>r\}$, the vector field $-G/G'$ is holomorphic there. For $z\notin E_s$, we consider the initial value problem
\[
\gamma_z'(u)=-\frac{G(\gamma_z(u))}{G'(\gamma_z(u))},
\qquad \gamma_z(0)=z.
\]
Local existence, uniqueness and continuous dependence hold. Along the solution, $G(\gamma_z(u))=e^{-u}G(z)$, so the stopping time is $T(z)=\log(|G(z)|/s)$. On any existing interval $0\le a<b\le T(z)$,
\begin{equation}\label{eq:flow-length}
\int_a^b\frac{|\gamma_z'(u)|}{1-|\gamma_z(u)|^2}\,du
\le\frac1c\int_{e^{-b}|G(z)|}^{e^{-a}|G(z)|}\frac{dv}{1-v^2}.
\end{equation}
If the maximal existence time were $\beta\le T(z)$, this estimate would make $\gamma_z(u)$ Cauchy in the complete disk metric as $u\to\beta$. Its limit $p\in\D$ would satisfy $|G(p)|=e^{-\beta}|G(z)|\ge s>r$. Local existence near $p$ would extend the solution, a contradiction. Thus the solution reaches level $s$.

For $0\le\lambda\le1$, we define
\[
H_s(z,\lambda)=
\begin{cases}
z,&z\in E_s,\\
\gamma_z(\lambda T(z)),&z\notin E_s.
\end{cases}
\]
Continuous dependence proves joint continuity away from $\{|G|=s\}$. At that level, \eqref{eq:flow-length} gives the uniform estimate
\[
d_\D(z,H_s(z,\lambda))\le c^{-1}\Delta_s(z)\to0
\qquad (|G(z)|\to s\text{ from above}),
\]
which proves joint continuity there as well. Hence $H_s$ is a strong deformation retraction, with $\rho_s=H_s(\cdot,1)$ mapping $\D$ onto $E_s$. In particular, $E_s$ is nonempty and contractible, as a retract of the disk. For $r<t<1$, we choose $r<s<t$. The homotopy preserves $\Omega_t(G)$, since $|G|$ never increases, and retracts it onto $E_s$. Thus $\Omega_t(G)$ is contractible too.

The upper bound in \eqref{eq:distance} follows by using $\rho_s(z)$. For $y\in E_s$, Schwarz-Pick and the triangle inequality give
\[
d_\D(z,y)\ge d_\D(G(z),G(y))
\ge\arctanh|G(z)|-\arctanh s.
\]
The infimum over $y$ gives the lower bound. \end{proof}

\begin{corollary}[Derivative criterion and connectedness threshold] \label{cor:derivative-criterion} A nonconstant inner function $\Theta$ is one-component if and only if there are $0\le r<1$ and $c>0$ such that $\eta_\Theta(z)\ge c$ whenever $|\Theta(z)|>r$. If $R=R_*(\Theta)<\infty$, then $\Omega_t(\Theta)$ is connected whenever
\begin{equation}\label{eq:threshold}
 q_R:=\sqrt{\frac{R^2-1}{R^2+1}}<t<1.
\end{equation}
In particular, when $R=1$, every $\Omega_t(\Theta)$ with $0<t<1$ is connected. \end{corollary}

\begin{proof} Necessity follows from Theorem~\ref{thm:main} and Proposition~\ref{prop:rank-one-derivative}; sufficiency follows from Lemma~\ref{lem:geometry}.

For the threshold, we fix $q_R<t<1$ and choose $q_R<r<t$. Equation~\eqref{eq:main-derivative} gives a positive uniform lower bound for $\eta_\Theta$ on $\{|\Theta|>r\}$, so Lemma~\ref{lem:geometry} gives connectedness of $\Omega_t(\Theta)$. \end{proof}

For $R=R_*(\Theta)<\infty$ and $q_R<s<1$, the flow in Lemma~\ref{lem:geometry}, using the bound $\eta_\Theta\ge c_s:=\sqrt{1-R^2(1-s^2)/(1+s^2)}$ along its stopped trajectories, gives a strong deformation retraction onto $E_s=\{|\Theta|\le s\}$ and \eqref{eq:distance} with $c=c_s$.

For the Blaschke product of degree two $\Theta_a(z)=(a-z^2)/(1-az^2)$, $0<a<1$, introduced above, we can compute the connectedness threshold directly. We let $u=z^2$. Then
\[ |\Theta_a(z)|<t  \text{ if and only if } |u-a|<t|1-au|. \]
The inequality on the right describes an open disk in the $u$-plane. For $0<t<a$, this disk avoids the origin; at $t=a$, it is tangent to the origin, which remains excluded because the inequality is strict. Thus, for every $0<t\le a$, the two branches of the square root map the disk onto exactly two components of $\Omega_t(\Theta_a)$. For $a<t<1$, the disk contains the origin. Since the disk $\{u:|u-a|<t|1-au|\}$ is convex and contains the origin, $u=z^2$ belonging to this disk implies that $\lambda^2u$ also belongs to it for every $0\le\lambda\le1$. Thus $\lambda z\in\Omega_t(\Theta_a)$, and the segment from the origin to $z$ lies in $\Omega_t(\Theta_a)$. Hence $\Omega_t(\Theta_a)$ is star-shaped about the origin and, in particular, connected. Thus,
\[ \Omega_t(\Theta_a)\text{ is connected} \quad\Longleftrightarrow\quad a<t<1. \]

\section{A quadratic criterion for the one-component property} \label{sec:quadratic-geometry}

The one-component property concerns the topology of a sublevel set. We now characterize it by the rate at which the Schwarz--Pick inequality approaches equality as the values of $\Theta$ approach the unit circle. Proposition~\ref{prop:rank-one-derivative} gives a bound for the defect $1-\eta_\Theta(z)^2$ that is linear in $1-|\Theta(z)|^2$. The annular covering property improves this to a quadratic bound, and the exponent is optimal.

\begin{theorem}[Geometric criterion]\label{thm:quadratic} A nonconstant inner function $\Theta$ belongs to $(\CLS)$ if and only if the quadratic bound
\begin{equation}\label{eq:quadratic}
\sup_{z\in\D}
\frac{1-\eta_\Theta(z)^2}{(1-|\Theta(z)|^2)^2}<\infty
\end{equation}
holds. The exponent $2$ in the denominator $(1-|\Theta(z)|^2)^2$ is optimal within the class $(\CLS)$: for every $p>2$, there exists $\Theta\in(\CLS)$ such that
\[
 \sup_{z\in\D}
 \frac{1-\eta_\Theta(z)^2}{(1-|\Theta(z)|^2)^p}
 =\infty.
\]\end{theorem}

The optimal exponent is already determined by the example $\Theta(z)=z^2$:
\[
\frac{1-\eta_\Theta(z)^2}{(1-|\Theta(z)|^2)^2}
=\frac{1}{(1+|z|^2)^4}\to\frac1{16}
\qquad (|z|\to1).
\]
For any larger exponent the corresponding quotient is unbounded.

For the quadratic criterion, we need to determine how quickly $\eta_\Theta$ approaches one as $|\Theta|$ approaches one. A positive lower bound already gives connectedness by Lemma~\ref{lem:geometry}. To obtain the rate in the other direction, we use the annular covering characterization of one-component inner functions \cite[Theorem 7.2]{IU}, based on Cohn's covering theorems \cite[Theorems 1.1 and 1.2]{Cohn}. This says that, above a sufficiently high level, $\Theta$ is a covering of an annulus. Here $\Theta$ maps onto the annulus, and each point of the annulus has a neighborhood $U$ such that $\Theta$ maps every component of $\Theta^{-1}(U)$ biholomorphically onto $U$. Hence we can continue a local inverse branch along any path in the annulus. We use invariance of the hyperbolic metric under coverings to compare the annulus and disk densities, and expansion of the resulting lower bound for $\eta_\Theta$ gives the quadratic term in the theorem.

A metric density $\lambda_V$ specifies the length element $\lambda_V(z)|dz|$ on a domain $V$. We use the hyperbolic metric facts in \cite[Theorems 10.4 and 10.5 and Section 12.2]{BM}, with densities divided by two to match our normalization.

\begin{proof}[Proof of Theorem~\ref{thm:quadratic}] If \eqref{eq:quadratic} holds, then $\eta_\Theta\ge1/\sqrt2$ whenever $|\Theta|$ is sufficiently close to one, so Corollary~\ref{cor:derivative-criterion} gives $\Theta\in(\CLS)$.

Conversely, suppose $\Theta\in(\CLS)$. By \cite[Theorem 7.2]{IU}, there is $0<r<1$ such that $\Theta$ is a covering over $A_r=\{\zeta:r<|\zeta|<1\}$. For a component $V$ of $\Theta^{-1}(A_r)$, we write $\lambda_V$ for its hyperbolic density. Covering invariance and monotonicity give
\[
\lambda_{A_r}(\Theta(z))|\Theta'(z)|
=\lambda_V(z)\ge\lambda_\D(z),\qquad z\in V.
\]
We put $L=\log(1/r)$ and $u=\log(1/|\Theta(z)|)$. The explicit annulus density \cite[Section 12.2]{BM} and $\lambda_\D(\zeta)=1/(1-|\zeta|^2)$ give
\begin{equation}\label{eq:annulus-bound}
\eta_\Theta(z)\ge
\frac{L}{\pi}\frac{\sin(\pi u/L)}{\sinh u}
=1-\frac{1+\pi^2/L^2}{6}u^2+O(u^4).
\end{equation}
As $u\to0$, we have $1-|\Theta(z)|^2=2u+O(u^2)$. Thus the quotient in \eqref{eq:quadratic} is bounded near $|\Theta|=1$. On the remaining set, its denominator is bounded below and its numerator is at most one, which proves \eqref{eq:quadratic}. Together with the example above, this proves the theorem. \end{proof}

\section{Sharp quantitative symbol bounds}\label{sec:quantitative}

The qualitative characterization says that testing the interior rank-one family suffices to obtain bounded symbols for all operators in $\TT(\Theta)$. We now determine the loss in passing from the test constant $R_*$ to the general constant $M_*$. The upper bound has a logarithmic factor, and finite Blaschke products show that this factor is necessary even in finite dimensions.

We use $X\sim Y$ for two-sided bounds with positive constants independent of the parameters; $\lesssim$ and $\gtrsim$ denote the corresponding one-sided bounds.

\begin{samepage}
\begin{theorem}[Sharp comparison of the symbol constants] \label{thm:sharp-rank-one} There is an absolute constant $C>0$ such that, whenever $R_*(\Theta)<\infty$,
\begin{equation}\label{eq:sharp-rank-one}
 R_*(\Theta)\le M_*(\Theta)
 \le C R_*(\Theta)\log(e+R_*(\Theta)).
\end{equation}
For $N\ge2$, put
\[
 a_N=1-\frac1N,\qquad
 \Theta_N(z)=\frac{a_N-z^N}{1-a_Nz^N}.
\]
With absolute comparison constants,
\begin{equation}\label{eq:sharp-family-intro}
 R_*(\Theta_N)\sim\frac{N}{\log N},\qquad
 M_*(\Theta_N)\sim N.
\end{equation}
Thus the upper bound cannot be replaced by $o(R_*\log R_*)$ as $R_*\to\infty$. \end{theorem}
\end{samepage}

We put $R=R_*(\Theta)<\infty$. By \eqref{eq:main-derivative}, the least $L^2$ symbol calculation gives $\eta_\Theta\ge1/2$ whenever $|\Theta|>e^{-\ell}$, with $\ell=cR^{-2}$ for a small absolute constant $c>0$. The covering argument below and Lemma~\ref{lem:annular-symbol-bound} would then give only $M_*(\Theta)\le CR^2$. To improve this bound, we test a sum of three products of model space functions, and use the resulting hyperbolic derivative estimate with a logarithmic factor to obtain a wider annulus in the covering argument. We then estimate weighted polynomial norms for the finite Blaschke family in Theorem~\ref{thm:sharp-rank-one}.

\subsection{The upper bound}

\begin{lemma}[Hyperbolic derivative estimate with a logarithmic factor]\label{lem:log-derivative} There is an absolute constant $C$ such that, for every nonconstant inner function $\Theta$ with $R_*(\Theta)<\infty$,
\begin{equation}\label{eq:log-derivative-estimate}
 1-\eta_\Theta(w)
 \le C R_*(\Theta)(1-|\Theta(w)|^2)
       \log\frac{e}{1-|\Theta(w)|^2},
 \qquad w\in\D.
\end{equation}
\end{lemma}

The second estimate gives a symbol bound in terms of the width of a covering annulus.

\begin{lemma}[Symbols from an annular covering]\label{lem:annular-symbol-bound} There is an absolute constant $C$ with the following property. Here $\theta$ is a nonconstant inner function and $0<\ell\le1/8$. If
\[
 \theta:\theta^{-1}(A_\ell)\to A_\ell,
 \qquad A_\ell=\{\zeta:e^{-\ell}<|\zeta|<1\},
\]
is a covering map, then
\[
 M_*(\theta)\le \frac C\ell.
\]
\end{lemma}

We first combine these lemmas: the derivative bound gives the covering, whose width determines the symbol bound. We then prove the logarithmic estimate; the annular estimate is proved in Appendix~\ref{app:annular-symbol}.

\begin{proof}[Proof of the upper bound in Theorem~\ref{thm:sharp-rank-one}] We write $R=R_*(\Theta)$. Since symbol norm dominates operator norm, $R\ge1$. For a sufficiently small absolute constant $c_0>0$, we put
\[
 \ell=\frac{c_0}{R\log(e+R)}.
\]
Lemma~\ref{lem:log-derivative} implies $\eta_\Theta(z)\ge1/2$ whenever $|\Theta(z)|>e^{-\ell}$. In fact, on this set $s=1-|\Theta(z)|^2\le2\ell$, and
\[
 CRs\log(e/s)\le 2CR\ell\log(e/(2\ell))\le\frac12
\]
after decreasing $c_0$. Indeed, $R\ell\log(e/(2\ell))\le c_0(3+\log(1/c_0))$ for $c_0\le1/2$, uniformly in $R\ge1$.

We next show that $\Theta$ is a covering over $A_\ell$ by using the bound $\eta_\Theta\ge1/2$ to control the lengths of lifted paths. As in \cite[Proposition~3.1]{RS}, we use completeness of the disk metric, together with the relation between continuation of inverse branches and covering maps recalled in \cite[Section~1]{Eremenko}. Here the rank-one estimate determines the width $\ell$, and we include the continuation argument to keep track of this dependence. A piecewise smooth path $\gamma$ in that annulus has a local lift obtained by solving $z'=\gamma'/\Theta'(z)$. The derivative does not vanish above $A_\ell$, and
\[
 \frac{|z'|}{1-|z|^2}
 \le 2\frac{|\gamma'|}{1-|\gamma|^2}.
\]
A piecewise smooth path on a compact parameter interval has finite hyperbolic length in $A_\ell$, so completeness of the disk prevents its lift from escaping in finite time. Thus local inverses continue along every path in $A_\ell$. The image meets this annulus because $\Theta$ is inner, and continuation also shows that it covers the whole annulus. Lemma~\ref{lem:annular-symbol-bound} now gives
\[
 M_*(\Theta)\le\frac C\ell
 \le C'R_*(\Theta)\log(e+R_*(\Theta)).
\]
\end{proof}

\subsection{The hyperbolic derivative estimate with a logarithmic factor}

\begin{proof}[Proof of Lemma~\ref{lem:log-derivative}] Lemma~\ref{lem:automorphisms} and a unimodular rotation reduce us to $w=0$ and $\Theta(0)=a\in[0,1)$. We put
\[
 s=1-a^2,
 \qquad u=\frac{\Theta-a}{1-a\Theta}=zv,
 \qquad c=v(0).
\]
Then $u$ is inner, $v$ is inner or a unimodular constant, and $|c|=\eta_\Theta(0)$.

For any finite sum $H=\sum_j f_jg_j$ with $f_j,g_j\in K_\Theta$, the identity
\[
 \ip{R_0f_j}{C_\Theta g_j}=f_j(0)g_j(0),
 \qquad R_0=\widetilde k_0\otimes k_0,
\]
shows that
\begin{equation}\label{eq:rank-one-product-test}
 |H(0)|\le b_\Theta(R_0)\norm H_1.
\end{equation}
Indeed, a bounded symbol $\psi$ for $R_0$ gives $H(0)=\int_\T\psi z\overline\Theta H\,dm$; take the infimum over such symbols.

Multiplication by $h=\sqrt s/(1+au)$ maps $K_u$ unitarily onto $K_\Theta$ (the Crofoot multiplier, \cite[equation~(11)]{BBK}). For the introductory family $\Theta_a$, we have $u=-z^2$ and $c=0$. Since $h,zh\in K_{\Theta_a}$, the calculation using two products
\[
 H_a=\frac{h^2-a(zh)^2}{s}=\frac1{1-az^2}
\]
cancels one denominator factor. Here $H_a(0)=1$, and its $L^1$ norm is $L(a)=\int_\T|1+a\zeta|^{-1}\,dm(\zeta)$: the boundary map $-z^2$ preserves $m$. In the coordinate $\zeta=-e^{it}$ near $-1$, the integrand is comparable to $((1-a)^2+t^2)^{-1/2}$, whose integral grows logarithmically as $a\to1$.

For general $\Theta$, since $b_\Theta(R_0)\le R_*(\Theta)s$, we seek a product sum with value $1-|c|^2$ at zero and $L^1$ norm at most $C\log(e/s)$. We put
\begin{equation}\label{eq:log-product-test}
 F=\frac{1-\bar c v}{1+au},
\end{equation}
which has the required value. To use \eqref{eq:rank-one-product-test}, we must realize the same cancellation with products from $K_\Theta$.

The functions
\[
 1,\quad v,\quad p=z-\bar c u=P_{K_u}z
\]
belong to $K_u$. Since $u=zv$, we have $1-\bar c v+apv=(1-\bar c v)(1+au)$, and hence
\[
 F=\frac{h^2}{s}(1-\bar c v+apv).
\]
This is the required representation using three products.

Since $u(0)=0$, its boundary map preserves normalized Lebesgue measure. The estimate \eqref{eq:rank-one-product-test}, applied to $F$ with $\norm{R_0}=s$, gives
\[
 1-|c|^2\le R_*(\Theta)s(1+|c|)L(a).
\]
Thus $1-|c|\le R_*(\Theta)sL(a)$. For $a\le1/2$, the integral is bounded by two. For $a>1/2$, the comparison above gives
\[
 L(a)\le C\log\frac e{1-a^2}.
\]
This proves \eqref{eq:log-derivative-estimate}. \end{proof}

\subsection{Sharpness of the logarithmic loss} \label{subsec:sharp-log-loss}

The projection $P_{z^N}f$ is the Taylor polynomial of $f\in H^2$ of degree at most $N-1$. For an arbitrary inner function $\Theta$, the unique orthogonal decomposition
\begin{equation*}
 f=r+\Theta g,\qquad r=P_\Theta f\in K_\Theta,\quad g\in H^2
\end{equation*}
still identifies $r$ as a remainder, although it need not be a polynomial.

A finite Blaschke product of degree $N$ has the form $\Theta(z)=e^{i\alpha}\prod_{j=1}^N(z-a_j)/(1-\overline{a_j}z)$, where $\alpha\in\mathbb R$, $a_j\in\D$, and zeros are counted with multiplicity. Its model space has dimension $N$. If the zeros are distinct, $P_\Theta f$ is the unique element of $K_\Theta$ agreeing with $f$ at all $N$ zeros.

For $N\ge2$ and $1/2\le a<1$, we use the finite Blaschke products
\[
 \Theta_{a,N}(z)=\frac{a-z^N}{1-az^N}.
\]
In this family, $R_*$ measures normalized point evaluation on polynomials equipped with a weighted $L^1$ norm. A lower bound for $M_*$ comes from comparing their unweighted and weighted $L^1$ norms. With $1-a=1/N$, these estimates give $M_*\gtrsim N$ and $R_*\sim N/\log N$.

The relation between polynomial factorization and bounded extensions of finite Toeplitz matrices is discussed in \cite{Volberg}.

\begin{proof}[Proof of sharpness in Theorem~\ref{thm:sharp-rank-one}] \emph{Step 1: weighted polynomial duality.} We put
\[
 h(z)=\frac{\sqrt{1-a^2}}{1-az^N}.
\]
On $\T$, we put $W=|h|^2$. The functions $e_j(z)=h(z)z^j$, $0\le j<N$, form an orthonormal basis of $K_{\Theta_{a,N}}$. In this basis,
\[
 \langle A_\psi e_k,e_j\rangle
 =\int_\T\psi W z^{k-j}\,dm.
\]
Thus the truncated Toeplitz operators are precisely the Toeplitz matrices. Indeed, every Toeplitz matrix has a trigonometric polynomial symbol $\varphi$ in the basis $1,z,\ldots,z^{N-1}$, and $\varphi/W$ is a bounded symbol for the same matrix in the basis $(e_j)$.

For a polynomial $F$ of degree at most $2N-2$, we put
\[
 D_{a,N}(F)=\int_\T W|F|\,dm,
 \qquad S_N(r)=\sum_{j=0}^{N-1}r^{2j}.
\]
We claim that
\begin{equation}\label{eq:rank-polynomial-duality}
 R_*(\Theta_{a,N})
 =\sup_{\substack{F\ne0\\\deg F\le2N-2}}
 \frac{\displaystyle\sup_{w\in\D}
             |F(w)|/S_N(|w|)}{D_{a,N}(F)}.
\end{equation}
To see this, we write $\Theta=\Theta_{a,N}$ and let $C_\Theta f=\Theta\bar z\bar f$ be the model space conjugation. Products of two elements of $K_\Theta$ span exactly the functions $h^2F$ under consideration. The identities
\[
 \langle R_w f,C_\Theta g\rangle=f(w)g(w),
 \qquad
 z\bar\Theta h^2F=-Wz^{1-N}F\quad\text{on }\T
\]
show that a symbol for $R_w$ represents the functional
\[
 -Wz^{1-N}F\mapsto h(w)^2F(w)
\]
on this finite-dimensional subspace of $L^1(m)$. Conversely, any $L^\infty$ representation of this functional is a symbol for $R_w$, by the same pairing identity. By Hahn--Banach,
\[
 b_\Theta(R_w)
 =|h(w)|^2\sup_{\substack{F\ne0\\\deg F\le2N-2}}
             \frac{|F(w)|}{D_{a,N}(F)}.
\]
Since $\norm{R_w}=\norm{k_w}^2 =|h(w)|^2S_N(|w|)$, this proves \eqref{eq:rank-polynomial-duality}.

\smallskip\noindent\emph{Step 2: the rank-one symbol constant.} With $P_a(u)=(1-a^2)/|1-au|^2$, we put
\[
 V_a=\int_\T P_a(u)|1-u|\,dm(u),
 \quad
 J_a=\int_\T P_a(u)|\operatorname{Im}u|\,dm(u).
\]
We next prove the estimates, uniform in $N$,
\begin{equation}\label{eq:rank-uniform-family}
 \frac1{V_a}\le R_*(\Theta_{a,N})\le\frac4{J_a}.
\end{equation}
The lower bound follows from \eqref{eq:rank-polynomial-duality} with $F(z)=1-z^N$ and $w=0$, since $D_{a,N}(F)=V_a$.

For the upper bound, we write $F=P+z^NQ$, where $\deg P\le N-1$ and $\deg Q\le N-2$. For $u\in\T$, interpolation at the roots $\zeta_j^N=u$ gives
\[
 P(w)+uQ(w)
 =\frac1N\sum_{j=1}^N F(\zeta_j)
                      \sum_{k=0}^{N-1}(w\bar\zeta_j)^k.
\]
For $0\le r<1$, $\sum_{k=0}^{N-1}r^k\le2S_N(r)$. We integrate over $\T$ by first averaging over the $N$ roots of $\zeta^N=u$, and then integrating in $u$. This gives
\[
 \begin{split}
 D_{a,N}(F)
 &=\int_\T P_a(u)\frac1N
                 \sum_{\zeta^N=u}|F(\zeta)|\,dm(u)\\
 &\ge\frac1{2S_N(|w|)}
       \int_\T P_a(u)|P(w)+uQ(w)|\,dm(u).
 \end{split}
\]
For any $x,y\in\mathbb C$, pairing $u=e^{it}$ with $\bar u=e^{-it}$ and using the triangle inequality yields
\[
 \int_\T P_a(u)|x+uy|\,dm(u)
 \ge J_a\max\{|x|,|y|\}
 \ge\frac{J_a}{2}(|x|+|y|).
\]
The first inequality follows for $|y|$ by subtracting the two expressions inside the absolute values; the inequality for $|x|$ follows after multiplication by $\bar u$. Hence
\[
 D_{a,N}(F)\ge\frac{J_a}{4S_N(|w|)}|F(w)|,
\]
which proves \eqref{eq:rank-uniform-family}.

We put $\varepsilon=1-a$. The comparisons
\[
 P_a(e^{it})\sim\frac{\varepsilon}{\varepsilon^2+t^2}
 \quad (|t|\le\pi),
 \qquad |1-e^{it}|\sim |t|
\]
give
\begin{equation}\label{eq:rank-family-asymptotic}
 V_a\sim J_a\sim
 \varepsilon\log(1/\varepsilon),
 \qquad
 R_*(\Theta_{a,N})\sim
 \frac1{\varepsilon\log(1/\varepsilon)}.
\end{equation}
For $J_a$, we use $|\sin t|\sim|t|$ on $|t|\le\pi/2$; the remaining integral is $O(\varepsilon)$. All constants are absolute for $0<\varepsilon\le1/2$.

\smallskip\noindent\emph{Step 3: the general symbol constant.} To estimate the general constant from below, we use
\begin{equation}\label{eq:global-polynomial-lower}
 M_*(\Theta_{a,N})\ge
 \frac{\norm{F}_{L^1(m)}}{D_{a,N}(F)}
 \qquad (0\ne F,\ \deg F\le2N-2).
\end{equation}
Indeed, let $q=z^{1-N}F$ on $\T$ and set $\varphi=\bar q/|q|$, with value zero where $q=0$. The compression of multiplication by $\varphi$ to $\operatorname{span}\{1,z,\ldots,z^{N-1}\}$ gives a Toeplitz matrix $A$ of norm at most one. If $\psi$ is any symbol for this matrix on $K_{\Theta_{a,N}}$, equality of the matrix entries implies
\[
 \int_\T\psi Wq\,dm
 =\int_\T\varphi q\,dm
 =\int_\T|F|\,dm.
\]
The infimum over $\psi$ gives \eqref{eq:global-polynomial-lower}.

The weight $W$ concentrates near the $N$th roots of unity. To make the weighted norm small, we include the vanishing factor $1-z^N$; the remaining degree permits the geometric sum $(1-z^{N-1})/(1-z)$. Hence we choose
\[
 F_N(z)=(1-z^N)\frac{1-z^{N-1}}{1-z}.
\]
Its degree is $2N-2$, and its unweighted norm grows logarithmically: the comparison
\begin{equation}\label{eq:sharp-polynomial-comparison}
 \left||F_N(e^{it})|
 -\frac{|1-e^{iNt}|^2}{|1-e^{it}|}\right|
 \le |1-e^{iNt}|
\end{equation}
gives
\begin{equation}\label{eq:sharp-polynomial-l1}
 \norm{F_N}_{L^1(m)}\ge c\log N.
\end{equation}
To check the logarithmic term, the integral of the second expression in \eqref{eq:sharp-polynomial-comparison} is $\frac2\pi\log N+O(1)$: replace $|1-e^{it}|^{-1}$ by $|t|^{-1}$ on $[-\pi,\pi]$, at a bounded cost, and use $\int_1^N\cos t\,dt/t=O(1)$. The finitely many small values of $N$ are absorbed by reducing $c$.

For the weighted integral, we set
\[
 g_a(t)=P_a(e^{it})|1-e^{it}|^2.
\]
The bounds
\[
 g_a(t)\le C\varepsilon\quad(t\in\mathbb R),
 \qquad
 g_a(t)\le Ct^2/\varepsilon\quad(|t|\le\varepsilon)
\]
and \eqref{eq:sharp-polynomial-comparison} imply
\[
 \begin{split}
 D_{a,N}(F_N)
 &\le\int_{-\pi}^{\pi}
       \frac{g_a(Nt)}{|1-e^{it}|}\,\frac{dt}{2\pi}+V_a\\
 &\le C\left(
       \int_0^{\varepsilon/N}\frac{N^2t}{\varepsilon}\,dt
       +\int_{\varepsilon/N}^{\pi}\frac{\varepsilon}{t}\,dt
       \right)+V_a\\
 &\le C\varepsilon\left(1+\log\frac{N}{\varepsilon}\right).
 \end{split}
\]
We now take $\varepsilon=1/N$. Equations~\eqref{eq:global-polynomial-lower} and \eqref{eq:sharp-polynomial-l1} give $M_*(\Theta_N)\ge cN$, whereas \eqref{eq:rank-family-asymptotic} gives $R_*(\Theta_N)\sim N/\log N$. The upper bound $M_*(\Theta_N)\le CN$ follows from the upper estimate in Theorem~\ref{thm:sharp-rank-one}, proved above. Finally, $\log(e+R_*(\Theta_N))\sim\log N$, so
\[
 M_*(\Theta_N)\sim
 R_*(\Theta_N)\log(e+R_*(\Theta_N)).
\]
This proves the asserted sharpness. \end{proof}

Equation~\eqref{eq:rank-family-asymptotic} sharpens the lower bound for degree two, and shows that both constants tend to infinity for every fixed $N\ge2$ as $a\to1$. Sarason's earlier result also covers self-adjoint operators \cite[Sections~5 and~7]{Sarason2012}, answering \cite[Section~7, p.~509]{Sarason}.

\subsection{Failure of a reverse geometric estimate}

The geometric defect
\[
 Q(\Theta)=\sup_{z\in\D}
 \frac{1-\eta_\Theta(z)^2}{1-|\Theta(z)|^2}
\]
satisfies $Q(\Theta)\le R_*(\Theta)^2\le M_*(\Theta)^2$ by \eqref{eq:main-derivative}. Finiteness of $Q$ also characterizes the one-component property by Corollary~\ref{cor:derivative-criterion}. The next example shows that the corresponding reverse estimate for the constants need not hold.

\begin{proposition}[Failure of control by the geometric defect] \label{prop:quantitative-failure} There is no absolute constant $C$ such that
\[
 M_*(\Theta)^2\le C(1+Q(\Theta))
\]
for all one-component inner functions. This fails already for Blaschke products of degree two. \end{proposition}

\begin{proof} The introductory family $\Theta_a=\Theta_{a,2}$, $1/2\le a<1$, provides the example. Postcomposition by a disk automorphism preserves the hyperbolic derivative, so $\eta_{\Theta_a}(z)=2|z|/(1+|z|^2)$ and
\[
 \frac{1-\eta_{\Theta_a}(z)^2}{1-|\Theta_a(z)|^2}
 =\frac{(1-|z|^2)|1-az^2|^2}
        {(1-a^2)(1+|z|^2)^3}
 \le\frac1{1-a^2}.
\]
Equality holds at zero, giving $Q(\Theta_a)=1/(1-a^2)$. Equation~\eqref{eq:rank-family-asymptotic} with $N=2$ gives
\[
 M_*(\Theta_a)\ge R_*(\Theta_a)
 \sim\frac1{(1-a)\log(1/(1-a))}\qquad(a\to1).
\]
Consequently, for an absolute $c>0$ and $a$ sufficiently close to one,
\[
 \frac{M_*(\Theta_a)^2}{1+Q(\Theta_a)}
 \ge\frac{c}{(1-a)\log^2(1/(1-a))}\to\infty.
\]
\end{proof}

\section{Model space consequences and trace class counterexamples} \label{sec:model-space-consequences}

\subsection{Embedding, factorization and compact symbols}

For $1\le p<\infty$, the Hardy space $H^p$ consists of analytic functions whose integrals $\int_\T|f(r\zeta)|^p\,dm(\zeta)$ remain bounded as $0<r<1$ varies; the supremum is $\norm f_{H^p}^p$. For an inner function $\Phi$, we write $K_\Phi^p=H^p\cap\Phi\overline{zH^p}$, interpreting the intersection through boundary values. This agrees with the model space above when $p=2$. In the notation of \cite{BBK}, $\mathcal C_p(\Phi)$ is the class of finite complex Borel measures $\mu$ on $\T$ for which integration against $|\mu|$ is bounded on $K_\Phi^p$: for some constant $C$,
\[
 \int_\T|f|^p\,d|\mu|\le C\norm f_{H^p}^p,
\]
initially for functions continuous on the closed disk. These are the embedding classes in Theorem~\ref{thm:full-characterization}(vi). Products of two functions in $K_\Theta$ belong to $H^1\cap\bar z^{\,2}\Theta^2\overline{H^1}$, where the intersection is understood through boundary values. The factorization assertion gives a converse with norm control: every function in this space is a sum of four such products, and the sum of the factors' $H^2$ norm products is bounded by a constant times its $H^1$ norm.

\cite[Theorem~2.4 and Proposition~4.1]{BBK} proved that the bounded symbol property is equivalent to the embedding and factorization conditions in Theorem~\ref{thm:full-characterization}(vi)--(vii), including the representation using four products; see also \cite[Section~5]{GarciaRossSurvey}. Next, its equivalence with the compact symbol condition in (viii) was proved by O'Loughlin \cite[Theorem~3.2]{OLoughlin}. Theorem~\ref{thm:main} identifies the one-component property as their common geometric condition.

\subsection{Prescribed singular value decay without bounded symbols}\label{sec:robust-failure}

Earlier rank-one examples establish failure of the bounded symbol property for particular inner functions \cite[Theorem~5.3]{BCFMT}; the theorem below gives counterexamples on every non-one-component model space. The proof uses a Baire category argument \cite{BanachSteinhaus}.

For compact $K$, we write $s_n(K)$ for its singular values in nonincreasing order with multiplicity, appending zeros if needed. The notation $S_1$ denotes the trace class, with $\norm K_{S_1}=\sum_n s_n(K)$.

\begin{theorem}[Prescribed singular value decay] \label{thm:arbitrary-decay} 
Assume that $\Theta$ is not one-component. For every $\delta>0$ and every sequence of positive numbers $(\varepsilon_n)_{n\ge1}$, there exists a self-adjoint trace class operator $K\in\TT(\Theta)$ such that
\[
 \norm K<\delta,\qquad
 s_n(K)\le\varepsilon_n\quad(n\ge1),\qquad
 b_\Theta(K)=+\infty.
\]
\end{theorem}

The construction controls the singular values and the least bounded symbol norm separately. In a series $\sum_j a_jt_jT_j$, we first choose summable $a_j$ so that the singular value bounds hold for every $|t_j|\le1$. We then choose self-adjoint $T_j$ with rank at most two, $\norm{T_j}\le1$ and $a_jb_\Theta(T_j)>j$. Cancellation may reduce the least bounded symbol norm of the sum; Baire's theorem selects parameters for which the sum has no bounded symbol.

\begin{proof} We may replace $\varepsilon_n$ by $\min_{1\le k\le n}\varepsilon_k$ and thus assume that the sequence is nonincreasing. We put
\[
 a_j=2^{-j}\min\{\delta/2,\varepsilon_{2j}\},\qquad j\ge1.
\]
By Theorem~\ref{thm:main}, the normalized interior operators of rank one $R$ have unbounded least symbol norms $b_\Theta(R)$, although each has a bounded symbol. We write $R=\operatorname{Re}R+i\operatorname{Im}R$, where $\operatorname{Re}R=(R+R^*)/2$ and $\operatorname{Im}R=(R-R^*)/(2i)$. Both parts are self-adjoint truncated Toeplitz operators, have rank at most two and norm at most one, and have bounded symbols. Since
\[
 b_\Theta(R)\le
 b_\Theta(\operatorname{Re}R)+b_\Theta(\operatorname{Im}R),
\]
we can choose, for every $j$, a self-adjoint operator $T_j$ of rank at most two such that
\begin{equation}\label{eq:large-selfadjoint-cost}
 \norm{T_j}\le1,\qquad b_\Theta(T_j)<\infty,
 \qquad a_jb_\Theta(T_j)>j.
\end{equation}

The set $\mathcal Q=[-1,1]^{\mathbb N}$ carries its compact metrizable product topology. We define
\[
 F(t)=\sum_{j\ge1}a_jt_jT_j,\qquad t=(t_j)\in\mathcal Q.
\]
Convergence in this topology means convergence in every coordinate. A basic neighborhood restricts only finitely many coordinates, and we use one of the remaining free coordinates to test any proposed uniform symbol bound. The series converges uniformly in trace norm, because $\norm{T_j}_{S_1}\le2$ and $\sum_j a_j<\infty$. Thus $F$ is continuous into the trace class, and its values are self-adjoint compact truncated Toeplitz operators; here we use the closedness of $\TT(\Theta)$ in the operator norm.

For integers $m\ge1$, the sets
\[
 E_m=\{t\in\mathcal Q:b_\Theta(F(t))\le m\}
\]
are closed by \eqref{eq:cost-sublevels}. None has interior. Otherwise, some $E_m$ would contain a nonempty basic open set, whose restrictions involve only finitely many coordinates. We choose a free coordinate $j>m$ and two points $t^+,t^-$ in this set which agree except for $t_j^+=1$ and $t_j^-=-1$. Subadditivity gives
\[
 2a_jb_\Theta(T_j)
  =b_\Theta(F(t^+)-F(t^-))\le2m,
\]
contrary to \eqref{eq:large-selfadjoint-cost}. The Baire category theorem supplies $t\notin\bigcup_m E_m$. We put $K=F(t)$.

The norm bound follows from $\norm K\le\sum_j a_j\le\delta/2$. For the singular values, we truncate the series before its $j$th term. The resulting operator has rank at most $2j-2$, so the approximation characterization of singular values yields
\[
 \begin{split}
 s_{2j-1}(K)
 &\le\left\|\sum_{k\ge j}a_kt_kT_k\right\|
 \le\sum_{k\ge j}a_k\\
 &\le 2^{1-j}\min\{\delta/2,\varepsilon_{2j}\}
 \le\varepsilon_{2j}.
 \end{split}
\]
Hence both $s_{2j-1}(K)\le\varepsilon_{2j-1}$ and $s_{2j}(K)\le\varepsilon_{2j}$. The choice of $t$ gives $b_\Theta(K)=+\infty$. \end{proof}

\begin{remark}\label{rem:BC-connection}
The search for symbol criteria for boundedness also motivated our earlier work on Toeplitz operators on the Bargmann--Fock space $F^2(\mathbb C^n)$ \cite{LooiBC}. There we constructed for each dimension $n\ge1$ a symbol $g$ satisfying $gk_a\in L^2$ for every reproducing kernel $k_a$ such that $T_g: F^2(\mathbb C^n)\to F^2(\mathbb C^n)$ is Hilbert--Schmidt, hence compact, but the heat transform of $g$ at the critical time is unbounded. This example disproves both the Berger--Coburn boundedness conjecture \cite{BergerCoburn} and the compactness conjecture of Bauer, Coburn and Isralowitz \cite[Conjecture~2]{BCI}, which proposed (scalar) heat transform criteria for boundedness and compactness, respectively. The converse boundedness implication also fails: boundedness of the critical heat transform does not imply boundedness of $T_g$ \cite{LooiBC}.

Theorem~\ref{thm:arbitrary-decay} gives a related obstruction for compact operators on model spaces: they may have no bounded symbol representative despite arbitrarily rapid singular value decay. In the Fock example, the failed condition is boundedness of the critical heat transform $H_{1/4}g$, where $H_t$ denotes heat evolution for $t>0$.
\end{remark}

\subsection{Generic failure and compact approximation}

The same category argument describes the size of the set of operators without bounded symbols. For $1\le p<\infty$, the Schatten class $S_p$ consists of compact operators whose singular values satisfy $\sum_n s_n(A)^p<\infty$, with norm $\norm A_{S_p}=(\sum_n s_n(A)^p)^{1/p}$. We write $\mathcal K(K_\Theta)$ for the compact operators on $K_\Theta$.

\begin{samepage}
\begin{corollary}[Generic failure] \label{prop:generic-failure} Assume that $\Theta$ is not one-component. For each $1\le p<\infty$, operators without bounded symbols form a dense $G_\delta$ subset of
\[
 X_p=\TT(\Theta)\cap S_p
\]
in the Schatten $p$-norm. The same holds in $X_{\mathrm c}=\TT(\Theta)\cap\mathcal K(K_\Theta)$ with the operator norm. Both assertions also hold in the respective real spaces of self-adjoint operators. \end{corollary}
\end{samepage}

\begin{proof} Each indicated space is Banach, since $\TT(\Theta)$ is closed in the operator norm. For any one of these spaces $X$, we set $D_m=\{A\in X:b_\Theta(A)\le m\}$. These sets are closed by \eqref{eq:cost-sublevels}. Theorem~\ref{thm:arbitrary-decay} gives a self-adjoint $K\in\TT(\Theta)\cap S_1$ without a bounded symbol; this $K$ belongs to every indicated space. If some $D_m$ contained a ball about $A_0$, then $A_0$ and $A_0+tK$ would lie in it for sufficiently small nonzero real $t$. The difference of their bounded symbols, divided by $t$, would give a bounded symbol for $K$, a contradiction. Thus every $D_m$ has empty interior, and the Baire category theorem proves the assertions. \end{proof}

The approximation result by O'Loughlin \cite[Corollary~3.9]{OLoughlin} states that the compact operators with symbols in $\Theta C(\T)$ are dense in $X_{\mathrm c}$ in the operator norm for every inner function. Here $C(\T)$ denotes the continuous functions on $\T$. Together with Corollary~\ref{prop:generic-failure}, this shows that, when $\Theta$ is not one-component, compact operators admitting bounded symbols form a dense meagre $F_\sigma$ subset of $X_{\mathrm c}$.

Moreover, for any sequence $A_n\to A$ in operator norm with $b_\Theta(A_n)<\infty$ and $b_\Theta(A)=+\infty$, we have
\begin{equation}\label{eq:symbol-cost-divergence}
 b_\Theta(A_n)\to+\infty.
\end{equation}
Otherwise a subsequence would belong to one fixed closed sublevel set \eqref{eq:cost-sublevels}, forcing its limit $A$ into that set.

\section{Norms and factorization in de Branges spaces} \label{sec:spectral-applications}

Every equivalent Hilbert norm can be represented by a bounded positive invertible operator. When the original norm is an integral, we ask whether the new norm can also be represented by a bounded function multiplying the original measure. More precisely, if $\norm F_0^2=\int_{\mathbb R}|F|^2\,d\mu$, a real measurable function $w$ is a \emph{representing density} for a new norm if $\norm F_*^2=\int_{\mathbb R}w|F|^2\,d\mu$ for every $F$. We seek one function $w$ with $0<c\le w\le C<\infty$ almost everywhere that gives this identity throughout the space.

The identity $\int_\T w|f|^2\,dm=\ip{A_w^\Theta f}{f}$ for $f\in K_\Theta$, $w\in L^\infty(m)$ links this question to bounded symbols. For example, on $K_{z^2}$ take real $\varepsilon$ with $|\varepsilon|<1/2$ and $w_\varepsilon(\zeta)=1+2\varepsilon\operatorname{Re}\zeta$. Then
\[
 \begin{aligned}
 \norm{\alpha+\beta z}_\varepsilon^2
 &=\int_\T w_\varepsilon(\zeta)|\alpha+\beta\zeta|^2\,dm(\zeta)\\
 &=|\alpha|^2+|\beta|^2+2\varepsilon\operatorname{Re}(\alpha\bar\beta).
 \end{aligned}
\]
Here $w_\varepsilon\ge1-2|\varepsilon|>0$: one positive weight represents the change in the cross term for all $\alpha,\beta$. In de Branges spaces, the norm changes of interest preserve the conjugation and zero-reflection operations defining the space. We call these changes admissible. Under the model space identification, their representing operators are exactly the positive invertible truncated Toeplitz operators.

\subsection{De Branges spaces and admissible norms}\label{subsec:de-branges-norms}

The upper half plane is $$\mathbb C_+=\{z\in\mathbb C:\operatorname{Im}z>0\}.$$ The space $H^p(\mathbb C_+)$, for $1\le p<\infty$, consists of holomorphic functions $h$ on $\mathbb C_+$ with $\sup_{y>0}\int_{\mathbb R}|h(x+iy)|^p\,dx<\infty$. For an entire function $F$, we let $F^\#(z) := \overline{F(\bar z)}$. An entire function $E$ is Hermite-Biehler if $|E^\#(z)|<|E(z)|$ on $\mathbb C_+$. We assume that $E$ has no real zeros and define
\[
 \begin{gathered}
 \mathcal H(E)=\{F\text{ entire}:
          F/E,F^\#/E\in H^2(\mathbb C_+)\},\\
 \norm F_E^2=\int_{\mathbb R}|F(x)/E(x)|^2\,dx.
 \end{gathered}
\]
Here the original norm is $\norm{\cdot}_E$, with measure $\mu=|E(x)|^{-2}\,dx$. The quotient $\Theta_E=E^\#/E$ is inner in the upper half plane. We call it one-component when $\{z\in\mathbb C_+:|\Theta_E(z)|<t\}$ is connected for some $0<t<1$. This is the same notion as in the disk after a conformal change of variable.

\begin{samepage}
An equivalent Hilbert norm $\norm{\cdot}_*$ on the same entire functions is called \emph{admissible} if two operations preserve it: conjugation, so $\norm{F^\#}_*=\norm F_*$, and reflection of a nonreal zero. More precisely, if $F(a)=0$ and $a\notin\mathbb R$, the function
\[
 F_a(z)=\frac{z-\bar a}{z-a}F(z)
\]
\end{samepage}
is entire, and we require $\norm{F_a}_*=\norm F_*$. The apparent pole at $a$ is removed by the zero of $F$. On the real line the multiplier has modulus one, which explains why this operation preserves norms defined by weights on the real line. We measure the size of a norm change by

\begin{equation}\label{eq:norm-perturbation}
 d=\sup_{F\ne0}\left|\frac{\norm F_*^2}{\norm F_E^2}-1\right|.
\end{equation}

Every equivalent Hilbert norm determines a unique bounded positive invertible operator $A$ such that $\norm F_*^2=\ip{AF}{F}_E$. The original norm corresponds to $A=I$.

\subsection{Admissible norms and truncated Toeplitz operators}\label{subsec:norm-correspondence} We assume $\mathcal H(E)\ne\{0\}$. Division by $E$ gives a unitary map $U_E:\mathcal H(E)\to K_{\Theta_E}$, where
\[
 K_{\Theta_E}=H^2(\mathbb C_+)\cap\bigl(\Theta_EH^2(\mathbb C_+)\bigr)^\perp.
\]
For the disk model of $\mathcal H(E)$, we use the Cayley transform \[ \chi(z)=\frac{z-i}{z+i}, \] which maps $\mathbb C_+$ conformally onto $\mathbb D$, and put $\theta=\Theta_E\circ\chi^{-1}$. The map
\[
 (Wh)(z)=\frac{h(\chi(z))}{\sqrt\pi\,(z+i)}
\]
is unitary from $K_\theta$ onto $K_{\Theta_E}$. We define half plane truncated Toeplitz operators by unitary conjugation with $W$. For bounded symbols, change of variables gives
\begin{equation}\label{eq:spectral-covariance}
 WA_\varphi^\theta W^*=A^+_{\varphi\circ\chi},
 \qquad
 \norm{\varphi\circ\chi}_{L^\infty(\mathbb R)}
       =\norm\varphi_{L^\infty(m)},
\end{equation}
where $A_w^+$ first multiplies $f\in K_{\Theta_E}$ by $w$ and then projects $wf$ orthogonally onto $K_{\Theta_E}$.

Point evaluations remain bounded under an equivalent norm, so admissibility ensures that the de Branges axioms hold for $\norm{\cdot}_*$ \cite{deBranges1959}.

\begin{samepage}
\begin{proposition}[Correspondence between norms and operators] \label{prop:spectral-norm-correspondence} Admissible norms on $\mathcal H(E)$ correspond bijectively to positive invertible truncated Toeplitz operators on $K_{\Theta_E}$, through
\[
 \norm F_A^2=\ip{AU_EF}{U_EF}.
\]
For $\norm{\cdot}_A$, the relative change in \eqref{eq:norm-perturbation} is $d=\norm{A-I}$. \end{proposition}
\end{samepage}

Under the Cayley transform, reflection of the zero at $-i$ across the real axis, given by multiplication by $(z-i)/(z+i)$, corresponds in the disk model to multiplication by the coordinate function $\zeta$, that is, $h(\zeta)\mapsto \zeta h(\zeta)$. This turns admissibility into Sarason's criterion for truncated Toeplitz operators. Conversely, the quadratic forms of truncated Toeplitz operators depend only on $|F/E|^2$ on $\mathbb R$, which both norm symmetries preserve.

\begin{proof} Let $T$ be the operator representing $\norm{\cdot}_*$, and let $A=U_ETU_E^{-1}$ be its unitary conjugate on $K_{\Theta_E}$. Then $\norm{\cdot}_*=\norm{\cdot}_A$. The quadratic-form formula for the norm of $A-I$ gives $d=\norm{A-I}$.

The functional representation in \cite[Theorem~2.3]{BBK} says that the form $\ip{Bh}{h}$ of a disk truncated Toeplitz operator is the value of a well-defined functional at the boundary function $|h|^2$. The same holds after unitary conjugation by $W$. Conjugation and reflection of nonreal zeros preserve $|F(x)/E(x)|$ on $\mathbb R$, and hence preserve $\norm{\cdot}_A$ when $A$ is a truncated Toeplitz operator.

Conversely, suppose that $\norm{\cdot}_A$ is admissible and put $B=W^*AW$. If $h,\zeta h\in K_\theta$, the entire functions $F=U_E^{-1}Wh$ and $G=U_E^{-1}W(\zeta h)$ satisfy $G=\chi F$. Since $G$ is entire, $F(-i)=0$. Invariance of the norm under reflection of the zero at $-i$ gives
\[
 \ip{B(\zeta h)}{\zeta h}=\ip{Bh}{h}.
\]
Sarason's shift invariance criterion \cite[Theorem~8.1]{Sarason} now gives $B\in\TT(\theta)$. \end{proof}

\subsection{Representing densities}

To obtain a positive density, we represent the small difference $A-I$ by a real symbol $b$, so that $w=1+b>0$ once $\norm b_\infty<1$.

\begin{theorem}[Bounded densities for de Branges norms] \label{thm:spectral-dichotomy} 
Assume $\mathcal H(E)\ne\{0\}$. The function $\Theta_E$ is one-component if and only if every sufficiently small admissible perturbation of $\norm{\cdot}_E$ has a bounded real representing density. In that case, put
\[
 M_E=M_*(\theta),\qquad R_E=R_*(\theta),
 \qquad \theta=\Theta_E\circ\chi^{-1}.
\]
Then $1\le M_E\le C R_E\log(e+R_E)$, with an absolute constant $C$, and every admissible norm $\norm{\cdot}_*$ with relative change $d<M_E^{-1}$ admits a real density $w$ satisfying
\begin{equation}\label{eq:spectral-realization}
 \norm F_*^2=\int_{\mathbb R}w(x)|F(x)/E(x)|^2\,dx,
 \qquad \norm{w-1}_\infty\le M_Ed.
\end{equation}
In particular, $w\ge1-M_Ed>0$ almost everywhere.

If $\Theta_E$ is not one-component, then for every $0<\delta<1$ and every positive sequence $(\varepsilon_n)$ there is an admissible equivalent norm represented by $I+K$, where $K$ is self-adjoint and trace class,
\[
 \norm K<\delta,\qquad s_n(K)\le\varepsilon_n\quad(n\ge1),
\]
and the norm has no bounded representing density. \end{theorem}

\begin{proof} We first assume that $\Theta_E$ is one-component. By Proposition~\ref{prop:spectral-norm-correspondence}, we write $\norm{\cdot}_*=\norm{\cdot}_A$. The definition of $M_E$ and \eqref{eq:spectral-covariance} give a bounded symbol $b$ of $A-I$ with
\[
 \norm b_\infty\le M_E\norm{A-I}.
\]
As shown in Section~\ref{sec:symbol}, the least bounded symbol norm is attained. Since $A-I$ is self-adjoint, replacing $b$ by its real part preserves the represented operator. The weight $w=1+b$ gives \eqref{eq:spectral-realization}. The bound for $M_E$ follows from Theorem~\ref{thm:sharp-rank-one}.

For the negative assertion, we conjugate the operator supplied by Theorem~\ref{thm:arbitrary-decay} using the unitary maps in Proposition~\ref{prop:spectral-norm-correspondence} to obtain an operator on $\mathcal H(E)$. The resulting $K$ has the same norm and singular values, and $\norm K<\delta<1$ makes $I+K$ positive and invertible, giving an admissible equivalent norm. A bounded representing density $w$ would give the bounded symbol $(w-1)\circ\chi^{-1}$ for the corresponding disk operator by \eqref{eq:spectral-covariance}, a contradiction. \end{proof}

\begin{remark}[Generic absence of bounded representing densities]
If $\mathcal H(E)\ne\{0\}$ and $\Theta_E$ is not one-component, then admissible norms without bounded representing densities are generic among small norm changes in the following sense.

For $\theta=\Theta_E\circ\chi^{-1}$, let $X$ be one of the spaces $X_p$ ($1\le p<\infty$) or $X_{\mathrm c}$ in Corollary~\ref{prop:generic-failure}, and let $X_{\mathrm{sa}}$ be its real subspace of self-adjoint operators. For every $0<r<1$, the operators in
$$
I+\{K\in X_{\mathrm{sa}}:\norm K_X<r\}
$$
whose associated norms have no bounded representing density form a dense $G_\delta$ subset of this ball, in the topology induced by $\norm{\cdot}_X$. This follows from Corollary~\ref{prop:generic-failure}, since a bounded representing density for the norm associated with $I+K$ exists exactly when $K$ has a bounded symbol.
\end{remark}

For related work on de Branges spaces equal as sets with different norms, see \cite[Section~7.5]{PoltoratskiOrder}.

\subsection{Discrete sampling norms}\label{subsec:discrete-sampling}

The square of a sampling norm is a weighted sum of squared point values. If the sampling norm is equivalent to the original norm, de Branges' theorem \cite{deBranges1959} already gives a positive integral representation; the question here is whether its density relative to the original weight can be bounded above and away from zero.

A weighted sum of squared samples at real points gives an admissible norm whenever it is equivalent to the original norm. Conversely, every admissible norm has a discrete sampling representation, by de Branges' theorem on orthogonal bases of reproducing kernels. This yields a sampling version of Theorem~\ref{thm:spectral-dichotomy}.

\begin{samepage}
\begin{corollary}[Bounded densities for sampling norms]\label{cor:discrete-spectral-weights} Assume $\mathcal H(E)\ne\{0\}$. The function $\Theta_E$ is one-component if and only if there exists $0<\delta_0<1$ with the following property: whenever a discrete set of real points $(\lambda_n)$ and positive weights $(\gamma_n)$ satisfy
\begin{equation}\label{eq:sampling-perturbation}
 (1-\delta_0)\norm F_E^2
 \le\sum_n\gamma_n|F(\lambda_n)|^2
 \le(1+\delta_0)\norm F_E^2,
\end{equation}
for every $F\in\mathcal H(E)$, there is a real measurable function $w$ with $0<c\le w\le C<\infty$ almost everywhere such that
\begin{equation}\label{eq:discrete-weight-realization}
 \sum_n\gamma_n F(\lambda_n)\overline{G(\lambda_n)}
 =\int_{\mathbb R}w(x)
       \frac{F(x)\overline{G(x)}}{|E(x)|^2}\,dx
\end{equation}
for all $F,G\in\mathcal H(E)$. The sampling set may be finite or countably infinite, with no finite accumulation point.

If $\Theta_E$ is one-component and the sampling bounds are $1\pm\delta$, where $0<\delta<M_E^{-1}$, then $w$ can be chosen with $\norm{w-1}_\infty\le M_E\delta$. In particular, $w\ge1-M_E\delta>0$ almost everywhere.

If $\Theta_E$ is not one-component, then for every $0<\delta<1$ and every positive sequence $(\varepsilon_n)$ there is a sampling norm with bounds $1\pm\delta$ for which \eqref{eq:discrete-weight-realization} has no bounded real solution $w$. The associated positive operator $S$, defined by
\[
 \ip{SF}{G}_E=\sum_n\gamma_n F(\lambda_n)\overline{G(\lambda_n)},
\]
can be chosen so that $S-I$ is self-adjoint and trace class, with
\[
 \norm{S-I}<\delta,
 \qquad s_n(S-I)\le\varepsilon_n\quad(n\ge1).
\]
\end{corollary}
\end{samepage}

The identity \eqref{eq:discrete-weight-realization} says that the discrete measure and $w(x)|E(x)|^{-2}\,dx$ give the same inner product on $\mathcal H(E)$, although the measures themselves need not be equal.

\begin{proof} The sampling inequalities make
\[
 \norm F_*^2=\sum_n\gamma_n|F(\lambda_n)|^2
\]
an equivalent Hilbert norm. Conjugation and reflection of a nonreal zero preserve $|F(\lambda_n)|$ at every real sampling point, so this norm is admissible. If its sampling bounds are $1\pm\delta$, its relative change in \eqref{eq:norm-perturbation} is at most $\delta$. When $\Theta_E$ is one-component, Theorem~\ref{thm:spectral-dichotomy} gives the asserted density and bound. Polarization gives \eqref{eq:discrete-weight-realization}. Any choice of $0<\delta_0<M_E^{-1}$ proves the forward implication.

For the converse and the final assertion, suppose that $\Theta_E$ is not one-component. Given $\delta$ and $(\varepsilon_n)$, Theorem~\ref{thm:spectral-dichotomy} provides an admissible norm represented by $I+K$, where $\norm K<\delta$, $K$ is self-adjoint and trace class, $s_n(K)\le\varepsilon_n$, and the norm has no bounded representing density. The same space of entire functions, equipped with this norm, is a de Branges space. We choose a discrete set of real points $(\lambda_n)$ for which its normalized reproducing kernels form an orthonormal basis, as supplied by de Branges' theorem; see \cite[Section~1.1]{MNO}. If $k_\lambda^*$ denotes the reproducing kernel in this new norm, Parseval's identity gives
\[
 \norm F_*^2
 =\sum_n\frac{|F(\lambda_n)|^2}{\norm{k_{\lambda_n}^*}_*^2}.
\]
Thus $\gamma_n=\norm{k_{\lambda_n}^*}_*^{-2}$ gives the required sampling norm, with $S=I+K$ and bounds $1\pm\delta$. A bounded real function satisfying \eqref{eq:discrete-weight-realization} would represent $\norm{\cdot}_*$, contradicting its construction. \end{proof}

\begin{remark}[Perturbing sampling weights] We assume that $\Theta_E$ is one-component and that
\[
 V:\mathcal H(E)\to\ell^2(J),\qquad
 VF=(\sqrt{\gamma_n}F(\lambda_n))_{n\in J},
 \quad \gamma_n>0,\quad\lambda_n\in\mathbb R,
\]
is unitary, where $J$ is finite or countably infinite. For a real sequence $(a_n)$ with $\delta=\sup_n|a_n|<M_E^{-1}$, the perturbed sampling norm
\[
 \norm F_*^2=\sum_{n\in J}(1+a_n)\gamma_n|F(\lambda_n)|^2
\]
is represented by the operator $V^*\operatorname{diag}(1+a_n)V$, so the relative change in the squared norm, as defined in \eqref{eq:norm-perturbation}, takes the value $d=\delta$. Corollary~\ref{cor:discrete-spectral-weights} gives a representing density $w_a$ with $\norm{w_a-1}_\infty\le M_E\delta$. \end{remark}

\subsection{Weak factorization of entire functions}
We define
\[
 \begin{gathered}
 \mathcal H^1(E^2)=\{F\text{ entire}:
       F/E^2,F^\#/E^2\in H^1(\mathbb C_+)\},\\
 \norm F_{1,E^2}=
       \int_{\mathbb R}\frac{|F(x)|}{|E(x)|^2}\,dx.
 \end{gathered}
\]
H\"older's inequality gives $GH\in\mathcal H^1(E^2)$ for $G,H\in\mathcal H(E)$, with $\norm{GH}_{1,E^2}\le\norm G_E\norm H_E$. For the converse factorization, we apply Theorem~\ref{thm:full-characterization}(vii).

\begin{samepage}
\begin{corollary}[Weak factorization of entire functions] \label{cor:entire-weak-factorization} The function $\Theta_E$ is one-component if and only if there is $C_E>0$ such that every $F\in\mathcal H^1(E^2)$ has a decomposition
\begin{equation}\label{eq:entire-weak-factorization}
 \begin{gathered}
 F=\sum_{j=1}^4G_jH_j,\qquad G_j,H_j\in\mathcal H(E),\\
 \sum_{j=1}^4\norm{G_j}_E\norm{H_j}_E
       \le C_E\norm F_{1,E^2}.
 \end{gathered}
\end{equation}
\end{corollary}
\end{samepage}

\begin{proof} Division by $E$ and $E^2$ identifies the factors and the target isometrically with $K_{\Theta_E}$ and
\[
 H^1(\mathbb C_+)\cap
       \Theta_E^2\overline{H^1(\mathbb C_+)},
\]
respectively, and preserves products. The intersection is understood through boundary values. For the inverse identification, the reflected Hardy function continues $E^2g$ across $\mathbb R$.

With $\chi,\theta,W$ fixed in Subsection~\ref{subsec:norm-correspondence}, the $H^1$ isometry
\[
 (W_1f)(z)=\frac{f(\chi(z))}{\pi(z+i)^2}
\]
satisfies $W_1(fg)=(Wf)(Wg)$. It carries the disk product space $H^1(\D)\cap\bar\zeta^{\,2}\theta^2\overline{H^1(\D)}$ onto the preceding half plane space. The equivalence of (i) and (vii) in Theorem~\ref{thm:full-characterization} proves the assertion. 
\end{proof}

For example, suppose that $\Theta_E$ is one-component. If a linear functional $\Lambda$ on $\mathcal H^1(E^2)$ satisfies $|\Lambda(GH)|\le B\norm G_E\norm H_E$ for all $G,H\in\mathcal H(E)$, then the corollary gives $|\Lambda(F)|\le BC_E\norm{F}_{1,E^2}$ for every $F\in\mathcal H^1(E^2)$. Thus a bound on products extends to $\mathcal H^1(E^2)$.

\begin{remark}[Phase conditions and sampling]\label{cor:phase} Let $\phi$ be the increasing phase of $E$, so $E(x)=|E(x)|e^{-i\phi(x)}$, and put $d\mu_\phi=\phi'(x)\,dx$. The measure $\mu_\phi$ is locally doubling if $\mu_\phi(2I)\le D\mu_\phi(I)$ whenever $\mu_\phi(2I)<1$, for some $D>0$; here $2I$ has the same centre as $I$ and twice its length. If $\Theta_E$ is a meromorphic Blaschke product with infinitely many zeros, Marzo, Nitzan and Olsen \cite[Section~2.1]{MNO} show that the following are equivalent: $\Theta_E$ is one-component; $\mu_\phi$ is locally doubling; and $|\phi''(x)|\le C\phi'(x)^2$ on $\mathbb R$ for some $C>0$. Thus each condition is equivalent to the bounded symbol property, the integral representation of the norm in Theorem~\ref{thm:spectral-dichotomy}, the sampling characterization in Corollary~\ref{cor:discrete-spectral-weights}, and the weak factorization in Corollary~\ref{cor:entire-weak-factorization}.

The sampling and interpolation theorems in \cite{MNO} assume the doubling inequality for all intervals, and they give density conditions on sampling and interpolation sequences. Corollary~\ref{cor:discrete-spectral-weights} concerns the integral representation of a sampling norm whose bounds are already close to one. Under the hypotheses above, local doubling is equivalent to the existence of a positive bounded representing density, bounded away from zero, for every sufficiently small sampling norm change.

The assumption of infinitely many zeros cannot be omitted from this equivalence. For example, take $E(z)=z+i$; then $\Theta_E(z)=(z-i)/(z+i)$ maps the upper half plane conformally onto $\D$, so its sublevel sets are connected. Its phase satisfies
\[
 \phi'(x)=\frac{1}{1+x^2},
 \qquad
 \frac{|\phi''(x)|}{\phi'(x)^2}=2|x|.
\]
This ratio is unbounded, so no constant $C$ can satisfy $|\phi''(x)|\le C\phi'(x)^2$ on $\mathbb R$. Thus the one-component property alone does not imply the phase estimate.
\end{remark}

\appendix \section{Quantitative symbol bounds from annular coverings}\label{app:annular-symbol}

Following Volberg \cite[Section~1]{Volberg} and Baranov, Bessonov and Kapustin \cite[Section~7]{BBK}, we construct positive majorants to prove Lemma~\ref{lem:annular-symbol-bound}, keeping track of the dependence on the annulus width $\ell$.

We use the space
\[
 X_\theta=\left\{
 f=\sum_n f_n\overline{g_n}:
 f_n,g_n\in K_\theta,\quad
 \sum_n\norm{f_n}_2\norm{g_n}_2<\infty\right\},
\]
with norm equal to the infimum of these sums. By the duality $X_\theta^*=\TT(\theta)$ from \cite[Theorem~2.3]{BBK} and Hahn--Banach, it suffices to prove
\begin{equation}\label{eq:annular-product-bound}
 \norm f_{X_\theta}\le\frac C\ell\norm f_1
 \qquad(f\in X_\theta).
\end{equation}
Indeed, this estimate extends the functional associated with $A\in\TT(\theta)$ to $L^1(m)$ with norm at most $C\ell^{-1}\norm A$, giving a bounded symbol with the same bound.

Positivity makes the estimate accessible: for nonnegative functions in $X_\theta$, the $X_\theta$ and $L^1$ norms coincide \cite[Proposition~4.1]{BBK}. For real-valued $f$, a majorant $G\in X_\theta$ with $G\ge|f|$ therefore gives $\norm f_{X_\theta}\le\norm G_1$ by applying this equality to $G+f$ and $G-f$. The covering supplies boundary intervals on which squared reproducing kernels yield such a majorant. Its integral is at most $C\ell^{-1}\norm f_1$, which is the required dependence on the annulus width.

\begin{proof}[Proof of Lemma~\ref{lem:annular-symbol-bound}] All constants below are absolute. By Lemma~\ref{lem:automorphisms}, we may precompose with a disk automorphism and assume $|\theta(0)|\le1/4$. Such a point exists because a nonconstant inner function has infimum modulus zero: otherwise $\log|\theta|$ would be bounded harmonic with boundary values zero. This normalization preserves the covering hypothesis and $M_*(\theta)$.

\smallskip\noindent\emph{Step 1: boundary continuation.} For $\alpha\in\T$, the covering hypothesis lifts $\alpha e^\zeta$ from the half-disk $|\zeta|<\ell$, $\Re\zeta<0$, to a map $\phi$ satisfying $\theta(\phi(\zeta))=\alpha e^\zeta$. The lift is injective, and it avoids zero because $|\theta(0)|<e^{-\ell}$. On the imaginary diameter, $|\phi|\to1$, since an interior cluster point would be mapped to $\T$. Reflection of $\log|\phi|$ therefore extends $\phi$ by
\[
 \phi(\zeta)=\frac1{\overline{\phi(-\bar\zeta)}}
 \qquad(\Re\zeta>0).
\]
The derivative cannot vanish on the diameter. At such a point, the first nonconstant Taylor term of a local logarithm of $\phi$ would have degree at least two, forcing $\log|\phi|$ to take both signs in the left half-disk, which contradicts $|\phi|<1$ there. If two points of the diameter had the same image, local inversion would give overlapping images of their left half-neighborhoods, contradicting injectivity on the left. Since the two half-disks map to opposite sides of $\T$, the extension is injective on $|\zeta|<\ell$.

A boundary point is called regular if $\theta$ extends analytically across it. Almost every $\xi\in\T$ is regular: if $\theta(r\xi)\to\alpha\in\T$, lift the logarithmic half-disk through a sufficiently late point of the radial tail. Uniqueness of lifting identifies the lift with that tail, and its extension through zero has value $\xi$ and nonzero derivative. Local inversion then continues $\theta$ across $\xi$.

At a regular point $t$, we put $v(t)=|\theta'(t)|$. We first obtain a uniform lower bound from Schwarz--Pick. With $a=|\theta(0)|\le1/4$, it gives $|\theta(rt)|\le(a+r)/(1+ar)$. Since $\theta$ is analytic and unimodular across the regular arc, letting $r\to1$ yields
\[
 v(t)=\lim_{r\to1}\frac{1-|\theta(rt)|^2}{1-r^2}
 \ge\lim_{r\to1}\frac{1-a^2}{(1+ar)^2}
 =\frac{1-a}{1+a}\ge\frac35.
\]

We next need the size of a disk on which $\theta$ continues. Koebe's quarter theorem states that, for an injective holomorphic map $h:\D\to\mathbb C$, the image $h(\D)$ contains the disk with centre $h(0)$ and radius $|h'(0)|/4$; see \cite[Theorem~8.5]{BM}. We apply it to $h(z)=\phi(\ell z)$, where $\phi(0)=t$. The derivative of $\theta(\phi(\zeta))=\theta(t)e^\zeta$ at zero gives $|\phi'(0)|=1/v(t)$. Thus $\phi(\{|\zeta|<\ell\})$ contains $D(t,\rho_t)$, where $\rho_t=\ell/(4v(t))$. On this disk, the inverse $q_t=\phi^{-1}$ gives the continuation $\theta(z)=\theta(t)e^{q_t(z)}$, with $|q_t(z)|<\ell$. In particular,
\begin{equation}\label{eq:annular-continuation}
 |\theta(z)|\le e^\ell
 \qquad (z\in D(t,\rho_t)).
\end{equation}

By Cauchy's estimate for the bounded analytic function $q_t$,
\[
 |q_t''(z)|\le\frac{C\ell}{\rho_t^2}
 \le\frac{C v(t)^2}{\ell}
 \qquad (|z-t|\le\rho_t/2).
\]
Since $|q_t'(t)|=v(t)$, integration shows that $|q_t'(z)|$ lies between $v(t)/2$ and $2v(t)$ whenever $|z-t|\le c\ell/v(t)$, for a sufficiently small absolute $c>0$. On $\T$, we have $|q_t'|=|\theta'|=v$, and $v$ is the derivative of the phase with respect to angular length. Consequently, the phase changes by at least a fixed multiple of $\ell$ before an arc from $t$ exits this smaller disk. Thus $v$ stays comparable to $v(t)$ on arcs through $t$ with sufficiently small phase increment relative to $\ell$.

On each proper maximal regular arc, the phase ranges over all of $\mathbb R$: a finite limit would let the reflected inverse continue $\theta$ across that endpoint.

\smallskip\noindent\emph{Step 2: phase intervals.} We fix a sufficiently small absolute $\kappa>0$ and partition the regular arcs into intervals $I_j$ with phase increments in $[\kappa\ell/2,\kappa\ell]$. If the whole circle is regular, we choose equal increments dividing the total phase change. The notation $|I_j|$ denotes angular length. For a chosen endpoint $t_j$, we put $v_j=v(t_j)$. By the comparison from Step~1,
\begin{equation}\label{eq:annular-phase-partition}
 |I_j|\sim\frac\ell{v_j},
 \qquad v(t)\sim v_j\quad(t\in I_j),
\end{equation}
with comparable lengths for adjacent intervals. Hence, small fixed enlargements $I_j^*$ have bounded overlap. We can choose an absolute $c_1>0$ so that disks of radius $c_1|I_j|$ centered on $I_j$ have radial projections in $I_j^*$ and lie in the continuation neighborhoods from \eqref{eq:annular-continuation}. (For related partitions, see \cite[Section~3]{Baranov2009} and \cite{HJK}.)

\smallskip\noindent\emph{Step 3: the sampling estimate.} For every finite sum $f=\sum_n f_n\overline{g_n}$ in $X_\theta$, we claim that
\begin{equation}\label{eq:annular-sampling}
 \sum_j |I_j|\sup_{I_j}|f|\le C\norm f_1.
\end{equation}
By taking real and imaginary parts, it suffices to prove this for real-valued $f$. Then $F=\bar z\theta f=\sum_n f_n C_\theta g_n$ belongs to $H^1$, with $\norm F_1=\norm f_1$, and reflection across a regular arc gives
\[
 F(z)=\frac{\theta(z)^2}{z^2}\overline{F(1/\bar z)}
 \qquad(|z|>1)
\]
in the continuation neighborhood.

The function $F^*$ denotes the radial maximal function. We apply the submean inequality on a disk of radius $c_1|I_j|$ centered at a point of $I_j$. Inside $\D$ we use $F^*$, whereas outside it we use the reflection formula, whose factor is bounded by $e^{2\ell}\le e^{1/4}$ by \eqref{eq:annular-continuation}. Radial integration over a length $O(|I_j|)$ gives
\[
 |I_j|\sup_{I_j}|f|
 \le C\int_{I_j^*}F^*\,dm.
\]
The $L^2$ boundedness of the Hardy--Littlewood maximal operator $\mathcal M$, applied to $F^*\le C[\mathcal M(|F|^{1/2})]^2$, gives $\norm{F^*}_1\le C\norm F_1$. We sum over $j$ and use bounded overlap to conclude \eqref{eq:annular-sampling}.

\smallskip\noindent\emph{Step 4: a positive majorant.} For real-valued $f$, we put $c_j=\sup_{I_j}|f|$. The kernel formula and \eqref{eq:annular-phase-partition} give
\[
 |k_{t_j}^\theta(t)|
 =\left|\frac{\theta(t)-\theta(t_j)}{t-t_j}\right|
 \ge c v_j\qquad(t\in I_j),
\]
with the value at $t_j$ understood by continuity. So for a sufficiently large absolute $C_0$,
\[
 G=C_0\sum_j c_j\frac{|k_{t_j}^\theta|^2}{v_j^2}
 \ge|f|\quad\text{almost everywhere}.
\]
Since $\norm{k_{t_j}^\theta}_2^2=v_j$, \eqref{eq:annular-phase-partition} and \eqref{eq:annular-sampling} give
\begin{equation}\label{eq:annular-majorant}
 \norm G_1=C_0\sum_j\frac{c_j}{v_j}
 \le\frac C\ell\sum_j|I_j|c_j
 \le\frac C\ell\norm f_1,
\end{equation}
which also proves convergence of the series in $X_\theta$.

By \cite[Proposition~4.1]{BBK}, every nonnegative $u\in X_\theta$ is a squared modulus of a function in $K_\theta$, so $\norm u_{X_\theta}=\norm u_1$. Applying this identity to $G+f$ and $G-f$, we obtain
\[
 \norm f_{X_\theta}
 \le\frac12\bigl(\norm{G+f}_1+\norm{G-f}_1\bigr)
 =\norm G_1\le\frac C\ell\norm f_1.
\]
Thus \eqref{eq:annular-product-bound} holds for real-valued finite sums. Decomposing into real and imaginary parts proves the same bound for complex-valued finite sums, after adjusting the absolute constant $C$. By density we obtain the bound for every $f\in X_\theta$, completing the proof.
\end{proof}

\subsection*{AI Usage Statement.}
GPT 5.6 was used to assist with bibliographic checks. GPT 5.6 was also used for light editing of language. I reviewed the bibliographic comments and language edits suggested by GPT 5.6 and take full responsibility for the paper.

\end{document}